\documentclass[11pt,a4paper,fleqn]{article}
\pdfoutput=1
\usepackage{amsthm}
\usepackage{amssymb}
\usepackage{amstext}
\usepackage{amsmath}
\usepackage{dsfont}
\usepackage{url}
\usepackage{hyperref}
\usepackage[margin = 40pt,font=small,labelfont=bf]{caption}
\usepackage{pdfsync}
\usepackage[mac]{inputenc}
\usepackage[T1]{fontenc}
\usepackage{color}
\usepackage{algorithm}
\usepackage{algorithmic}
\usepackage{enumerate}
\usepackage{bbm}
\usepackage{ae,aecompl}
\usepackage{pdfpages}
\usepackage{epstopdf}
\usepackage{graphicx}
\usepackage{epsfig}
\usepackage{caption}
\usepackage{subcaption}
\usepackage{pifont}
\usepackage{tabu}
\usepackage{pgfplots}
\pgfplotsset{compat=newest}
\usetikzlibrary{plotmarks}
\usepackage{grffile}

\newcommand{\cmark}{\ding{51}}%
\newcommand{\xmark}{\ding{55}}

\numberwithin{equation}{section}

\newcommand{\thefont}[2]{\fontsize{#1}{#2}\fontshape{n}\selectfont}
\newcommand{\1}{\rlap{\thefont{10pt}{12pt}1}\kern.16em\rlap{\thefont{11pt}{13.2pt}1}\kern.4em}

\def\argmin{\mathop{\rm arg \; min}\limits}%

\theoremstyle{plain}
\newtheorem{proposition}{Proposition}[section]
\newtheorem{corollary}{Corollary}[section]
\newtheorem{theorem}{Theorem}[section]
\newtheorem{lemma}{Lemma}[section]
\newtheorem{assumption}{Assumption}[section]

\theoremstyle{definition}
\newtheorem{definition}{Definition}[section]
\newtheorem{example}{Example}[section]

\theoremstyle{remark}

\newcommand{\R}{{\mathbb R}}
\newcommand{\C}{{\mathbb C}}

\newcommand{\N}{{\mathbb N}}
\newcommand{\Z}{{\mathbb Z}}

\renewcommand{\b}[1]{\boldsymbol{#1}}

\newcommand{\one}{\mathds{1}}

\def\argmin{\mathop{\rm arg \; min}\limits}%
\newcommand{\supp}{\mathop{\mathrm{supp}}}

\graphicspath{{./images/}}

\title{Approximation of integral operators using convolution-product expansions}

\author{ Paul Escande\footnote{D\'epartement d'Ing\'{e}nierie des Syst\`{e}mes Complexes (DISC), Institut Sup\'{e}rieur de l'A\'{e}ronautique et de l'Espace (ISAE), Toulouse, France, {\tt paul.escande@gmail.com}. This author is pursuing a Ph.D. degree supported by the MODIM project funded by the PRES of Toulouse University and Midi-Pyr\'en\`ees region.}  \hspace{1cm} Pierre Weiss\footnote{Institut des Technologies Avanc\'{e}es en Sciences du Vivant, ITAV-USR3505 and Institut de Math\'{e}matiques de Toulouse, IMT-UMR5219, CNRS and Universit\'{e} de Toulouse, Toulouse, France, {\tt pierre.armand.weiss@gmail.com}}  }

\date{\today}

\begin{document}

\maketitle

\thispagestyle{empty}

\begin{abstract}
We consider a class of linear integral operators with impulse responses varying regularly in time or space. 
These operators appear in a large number of applications ranging from signal/image processing to biology. 
Evaluating their action on functions is a computationally intensive problem necessary for many practical problems.
We analyze a technique called convolution-product expansion: the operator is locally approximated by a convolution, allowing to design fast numerical algorithms based on the fast Fourier transform. 
We design various types of expansions, provide their explicit rates of approximation and their complexity depending on the time varying impulse response smoothness. 
This analysis suggests novel wavelet based implementations of the method with numerous assets such as optimal approximation rates, low complexity and storage requirements as well as adaptivity to the kernels regularity. 
The proposed methods are an alternative to more standard procedures such as panel clustering, cross approximations, wavelet expansions or hierarchical matrices.
\end{abstract}



\section*{Acknowledgments} 

We began investigating a much narrower version of the problem in a preliminary version of \cite{escande2015sparse}.
An anonymous reviewer however suggested that it would be more interesting to make a general analysis and we therefore discarded the aspects related to convolution-product expansions from \cite{escande2015sparse}. We thank the reviewer for motivating us to initiate this research.
We thank Anh-Tuan Nguyen for giving a feedback and correcting some typos in a preliminary version of the paper.
We also thank Sandrine Anthoine, J\'er\'emie Bigot, Caroline Chaux, J\'er\^ome Fehrenbach, Hans Feichtinger, Clothilde M\'elot and Bruno Torr\'esani for fruitful discussions on related matters while elaborating a common project. In particular, the name product-convolution series is due to H. Feichtinger.

\section{Introduction} \label{sec:intro}

We are interested in the compact representation and fast evaluation of a class of space or time varying linear integral operators with regular variations.
Such operators appear in a large number of applications ranging from wireless communications \cite{roh2004efficient,hrycak2010low} to seismic data analysis \cite{griffiths1977adaptive}, biology \cite{gilad2006fast} and image processing \cite{sawchuk1972space}. 

In all these applications, a key numerical problem is to efficiently evaluate the action of the operator and its adjoint on given functions. This is necessary - for instance - to design fast inverse problems solvers. The main objective of this paper is to analyze the complexity of a set of approximation techniques coined \emph{convolution-product series}. 

We are interested in bounded linear integral operator $H:L^2(\Omega)\to L^2(\Omega)$ defined from a kernel $K$ by:
\begin{equation}
 Hu(x) = \int_{\Omega} K(x,y) u(y) \, dy. \label{eq:integralequation}
\end{equation}
for all $u\in L^2(\Omega)$, where $\Omega \subset \R^d$. 
Evaluating integrals of type \eqref{eq:integralequation} is a major challenge in numerical analysis and many methods have been developed in the literature. Nearly all methods share the same basic principle: decompose the operator kernel as a sum of low rank matrices with a multi-scale structure. This is the case in panel clustering methods \cite{hackbusch1989fast}, hierarchical matrices \cite{borm2003introduction}, cross approximations \cite{oseledets2010tt} or wavelet expansions \cite{beylkin1991fast}. 
The method proposed in this paper basically shares the same idea, except that the time varying impulse response $T$ of the operator is decomposed instead of the kernel $K$. The \emph{time varying impulse response} (TVIR) $T$ of $H$ is defined by:
\begin{equation}
 T(x,y) = K(x+y,y).
\end{equation}
The TVIR representation of $H$ allows formalizing the notion of regularly varying integral operator: the functions $T(x,\cdot)$ should be ``smooth'' for all $x\in \Omega$. 
Intuitively, the smoothness assumption means that two neighboring impulse responses should only differ slightly. 
Under this assumption, it is tempting to approximate $H$ locally by a convolution. 
Two different approaches have been proposed in the literature to achieve this. 
The first one is called \emph{product-convolution expansion of order $m$} and consists of approximating $H$ by an operator $H_m$ of type:
\begin{equation}\label{eq:productconvolution}
 H_m u = \sum_{k=1}^m w_k \odot (h_k \star  u),
\end{equation}
where $\odot$ denotes the standard multiplication for functions and the Hadamard product for vectors, and $\star$ denotes the convolution operator.
The second one, called \emph{convolution-product expansion of order $m$}, is at the core of this paper and consists of using an expansion of type:
\begin{equation}\label{eq:convolutionproduct}
 H_m u = \sum_{k=1}^m h_k \star (w_k \odot u).
\end{equation}
These two types of approximations have been used for a long time in the field of imaging (and to a lesser extent mobile communications and biology) and progressively became more and more refined \cite{trussell1992identification,nagy1998restoring,flicker2005anisoplanatic,gilad2006fast,hrycak2010low,hirsch2010efficient,miraut2012efficient,denis2015fast}. 
In particular, the recent work \cite{denis2015fast} provides a nice overview of existing choices for the functions $h_k$ and $w_k$ as well as new ideas leading to significant improvements. 
Many different names have been used in the literature to describe expansions of type \eqref{eq:productconvolution} and \eqref{eq:convolutionproduct} depending on the communities: sectional methods, overlap-add and overlap-save methods, piecewise convolutions, anisoplanatic convolutions, filter flow, windowed-convolutions,... The term product-convolution comes from the field of mathematics \cite{busby1981product}. We believe that it precisely describes the set of expansions of type \eqref{eq:productconvolution} and therefore chose this naming.
Now that convolution-product expansions have been described, natural questions arise: 
\begin{itemize}
 \item[i)] How to choose the functions $h_k$ and $w_k$?
 \item[ii)] What is the numerical complexity of evaluating products of type $H_m u$?
 \item[iii)] What is the resulting approximation error $\|H_m-H\|$, where $\|\cdot\|$ is a norm over the space of operators?
 \item[iv)] How many operations are needed in order to obtain an approximation $H_m$ such that  $\|H_m-H\|\leq \epsilon$?
\end{itemize}
Elements i) and ii) have been studied thoroughly and improved over the years in the mentioned papers. 
The main questions addressed herein are points iii) and iv). 
To the best of our knowledge, they have been ignored until now. 
They are however necessary in order to evaluate the theoretical performance of different convolution-product expansions and to compare their respective advantages precisely.

The main outcome of this paper is the following: under smoothness assumptions of type $T(x,\cdot) \in H^s(\Omega)$ for all $x\in \Omega$ (the Hilbert space of functions in $L^2(\Omega)$ with $s$ derivatives in $L^2(\Omega)$), most methods proposed in the literature - if implemented correctly - ensure a decay of type $\|H_m-H\|_{HS}= O(m^{-s})$, where $\|\cdot\|_{HS}$ is the Hilbert-Schmidt norm. Moreover, this bound cannot be improved uniformly on the considered smoothness class. By adding a support condition of type $\mathrm{supp}(T(x,\cdot)) \subseteq [-\kappa/2,\kappa/2]$, the bound becomes $\|H_m-H\|_{HS}= O(\sqrt{\kappa} m^{-s})$. More importantly, bounded supports allow reducing the computational burden. After discretization on $n$ time points, we show that the number of operations required to satisfy $\|H_m-H\|_{HS}\leq \epsilon$ vary from $O\left(\kappa^{\frac{1}{2s}}n\log_2(n)\epsilon^{-1/s}\right)$ to $O\left(\kappa^{\frac{2s+1}{2s}} n\log_2(\kappa n)\epsilon^{-1/s}\right)$ depending on the choices of $w_k$ and $h_k$.
We also show that the compressed operator representations of Meyer \cite{meyer1995wavelets} can be used under additional regularity assumptions.

The paper is organized as follows. In section \ref{sec:notation}, we describe the notation and introduce a few standard results of approximation theory. In section \ref{sec:preliminaries}, we precisely describe the class of operators studied in this paper, show how to discretize them and provide the numerical complexity of evaluating convolution-product expansions of type \eqref{eq:convolutionproduct}. 
Sections \ref{sec:linearsubspace} and \ref{sec:adaptive} contain the full approximation analysis for two different kinds of approaches called linear or adaptive methods. Section \ref{sec:additional} contains a summary and a few additional comments.

\section{Notation}\label{sec:notation}

Let $a$ and $b$ denote functions depending on some parameters. The relationship $a\asymp b$ means that $a$ and $b$ are equivalent, i.e. that there exists $0<c_1\leq c_2$ such that $c_1 a \leq b \leq c_2 a$. Constants  appearing  in  inequalities  will  be  denoted  by $C$ and may vary at each occurrence. If a dependence on a parameter exists (e.g. $\epsilon$), we will use the notation $C(\epsilon)$.

In most of the paper, we work on the unit circle $\Omega = \R\backslash \Z$ sometimes identified with the interval $\left[-\frac{1}{2},\frac{1}{2}\right]$. This choice is driven by simplicity of exposition and the results can be extended to bounded domains such as $\Omega=[0,1]^d$ (see section \ref{subsec:extensionhigherdim}).
Let $L^2(\Omega)$ denote the space of square integrable functions on $\Omega$. The Sobolev space $H^s(\Omega)$ is defined as the set of functions in $L^2(\Omega)$ with weak derivatives up to order $s$ in $L^2(\Omega)$. The $k$-th weak derivative of $u\in H^s(\Omega)$ is denoted $u^{(s)}$. The norm and semi-norm of $u\in H^s(\Omega)$ are defined by:
\begin{equation}
 \|u\|_{H^s(\Omega)} = \sum_{k=0}^s \|u^{(k)}\|_{L^2(\Omega)} \quad \textrm{and} \quad |u|_{H^s(\Omega)} = \|u^{(s)}\|_{L^2(\Omega)}. 
\end{equation}
The sequence of functions $(e_k)_{k \in \Z}$ where $e_k:x \mapsto \exp(-2i\pi k x)$ is a Hilbert basis of $L^2(\Omega)$ (see e.g. \cite{katznelson2004introduction}).
\begin{definition}
 Let $u\in L^2(\Omega)$ and $e_k:x \mapsto \exp(-2i\pi k x)$ denote the $k$-th Fourier atom. The Fourier series coefficients $\hat u[k]$ of $u$ are defined for all $k\in \Z$ by:
 \begin{equation}
  \hat u[k] = \int_{\Omega} u(x) e_k(x)\,dx.
 \end{equation}
\end{definition}

The space $H^s(\Omega)$ can be characterized through Fourier series. 
\begin{lemma}[Fourier characterization of Sobolev norms] \label{lem:Characterization_Fourier_Sobolev}
 \begin{equation}
    \|u\|^2_{H^s(\Omega)} \asymp \sum_{k \in \Z} |\hat{u}[k]|^2 (1 + |k|^2)^s.
 \end{equation}
\end{lemma}

\begin{definition}[B-spline of order $\alpha$] \label{def:b-spline}
Let $\alpha\in \N$ and $m\geq \alpha+2$ be two integers. The B-spline of order $0$ is defined by 
 \begin{equation}
B_{0,m} = \one_{[-1/(2m),1/(2m)]}.
 \end{equation} 
 The B-spline of order $\alpha \in \N^*$ is defined by recurrence by: 
 \begin{equation}
  B_{\alpha,m} = m B_{0,m} \star B_{\alpha-1,m} = m^{\alpha} \underbrace{B_{0,m} \star \hdots \star B_{0,m}}_{\alpha \textrm{ times}}.  
 \end{equation}
 
The set of cardinal B-splines of order $\alpha$ is denoted $\mathcal{B}_{\alpha,m}$ and defined by:
 \begin{equation}
  \mathcal{B}_{\alpha,m} = \left\{ f(\cdot) = \sum_{k=0}^{m-1} c_k B_{\alpha,m}(\cdot - k/m), \ c_k\in\R, \ 0\leq k \leq m-1  \right\}.
 \end{equation}
\end{definition}

In this work, we use Daubechies wavelet bases on $L^2(\R)$ \cite{daubechies1988orthonormal}. 
We let $\phi$ and $\psi$ denote the scaling and mother wavelets and assume that the mother wavelet $\psi$ has $\alpha$ vanishing moments, i.e. 
\begin{equation}
	\forall 0 \leq m < \alpha, \quad \int_{[0,1]} t^m \psi(t) dt = 0.
\end{equation}
Daubechies wavelets satisfy $\supp(\psi)=[-\alpha+1,\alpha]$, see \cite[Theorem 7.9, p. 294]{mallat1999wavelet}.
Translated and dilated versions of the wavelets are defined, for all $j > 0$ by
\begin{equation}\label{eq:defwavelets}
\psi_{j,l}(x) = 2^{j/2} \psi\left( 2^{j} x - l \right).
\end{equation}
The set of functions $(\psi_{j,l})_{j\in \N, l\in Z}$, is an orthonormal basis of $L^2(\R)$ with the convention $\psi_{0,l}=\phi(x-l)$. 
There are different ways to construct a wavelet basis on the interval $[-1/2,1/2]$ from a wavelet basis on $L^2(\R)$. Here, we use boundary wavelets defined in \cite{cohen1993wavelets}. We refer to \cite{daubechies1992ten,mallat1999wavelet} for more details on the construction of wavelet bases. This yields an orthonormal basis $(\psi_\lambda)_{\lambda\in \Lambda}$ of $L^2(\Omega)$, where 
\begin{equation}
  \Lambda = \left \{ (j,l), j \in \N, 0 \leq l \leq 2^j \right\}.
\end{equation}
We let $I_\lambda = \supp(\psi_\lambda)$ and for $\lambda\in \Lambda$, we use the notation $|\lambda|=j$.


Let $u$ and $v$ be two functions in $L^2(\Omega)$, the notation $u\otimes v$ will be used both to indicate the function $w\in L^2(\Omega\times\Omega)$ defined by
\begin{equation}
 w(x,y) = (u\otimes v)(x,y) = u(x) v(y),
\end{equation}
or the Hilbert-Schmidt operator $w:L^2(\Omega)\to L^2(\Omega)$ defined for all $f\in L^2(\Omega)$ by:
\begin{equation}
 w(f) = (u \otimes v) f = \langle u, f \rangle v.
\end{equation}
The meaning can be inferred depending on the context.
Let $H:L^2(\Omega)\to L^2(\Omega)$ denote a linear integral operators. 
Its kernel will always be denoted $K$ and its time varying impulse response $T$. 
The linear integral operator with kernel $T$ will be denoted $J$.

The following result is an extension of the singular value decomposition to operators.
\begin{lemma}[{Schmidt decomposition \cite[Theorem 2.2]{pinkus2012n} or \cite[Theorem 1 p. 215]{helemskii2006lectures}}]\label{lem:schmidt}
 Let $H:L^2(\Omega) \to L^2(\Omega)$ denote a compact operator. There exists two finite or countable orthonormal systems $\{e_1, \ldots \}$, $\{f_1, \ldots \}$ of $L^2(\Omega)$ and a finite or infinite sequence $\sigma_1 \geq \sigma_2 \geq \ldots$ of positive numbers (tending to zero if it is infinite), such that $H$ can be decomposed as:
\begin{equation}
 H = \sum_{k \geq 1} \sigma_k \cdot e_k \otimes f_k.
\end{equation}
\end{lemma}

A function $u\in L^2(\Omega)$ is denoted in regular font whereas its discretized version $\b{u}\in \R^n$ is denoted in bold font.
The value of function $u$ at $x\in \Omega$ is denoted $u(x)$, while the $i$-th coefficient of vector $\b{u}\in \R^n$ is denoted $\b{u}[i]$. Similarly, an operator $H : L^2(\Omega) \to L^2(\Omega)$ is denoted in upper-case regular font whereas its discretized version $\b{H} \in \R^{n \times n}$ is denoted in upper-case bold font.
\section{Preliminary facts} \label{sec:preliminaries}

In this section, we gather a few basic results necessary to derive approximation results.

\subsection{Assumptions on the operator and examples} \label{sec:regularity}

All the results stated in this paper rely on the assumption that the TVIR $T$ of $H$ is a sufficiently simple function. 
By simple, we mean that i) the functions $T(x,\cdot)$ are smooth for all $x\in \Omega$ and ii) the impulse responses $T(\cdot,y)$ have a bounded support or a fast decay for all $y\in \Omega$.

There are numerous ways to capture the regularity of a function. In this paper, we assume that $T(x,\cdot)$ lives in the Hilbert spaces $H^s(\Omega)$ for all $x\in \Omega$. This hypothesis is deliberately simple to clarify the proofs and the main ideas. 
\begin{definition}[Class $\mathcal{T}^s$] \label{def:class_op}
We let $\mathcal{T}^s$ denote the class of functions $T : \Omega \times \Omega \to \R$ satisfying the smoothness condition:
$T(x,\cdot) \in H^s(\Omega), \  \forall x \in \Omega$ and $\|T(x,\cdot)\|_{H^s(\Omega)}$ is uniformly bounded in $x$, i.e:
\begin{equation}
 \sup_{x\in \Omega}\|T(x,\cdot)\|_{H^{s}(\Omega)} \leq  C <+\infty.
\end{equation}
\end{definition}
Note that if $T\in \mathcal{T}^s$, then $H$ is a mere Hilbert-Schmidt operator since:
\begin{align}
 \|H\|_{HS}^2 &= \int_{\Omega} \int_{\Omega} K(x,y)^2\,dx\,dy \\
              &= \int_{\Omega} \int_{\Omega} T(x,y)^2\,dx\,dy \\ 
              &= \int_{\Omega} \|T(x,\cdot)\|_{L^2(\Omega)}^2\,dx <+\infty.
\end{align}

We will often use the following regularity assumption.
\begin{assumption}\label{assumption1}
 The TVIR $T$ of $H$ belongs to $\mathcal{T}^s$.
\end{assumption}

In many applications, the impulse responses have a bounded support, or at least a fast spatial decay allowing to neglect the tails. 
This property will be exploited to design faster algorithms. 
This hypothesis can be expressed by the following assumption.
\begin{assumption}\label{assumption2}
$T(x,y)=0, \ \forall |x|>\kappa/2$.
\end{assumption}

\subsection{Examples}

We provide 3 examples of kernels that may appear in applications.  
Figure \ref{fig:kernels_illustrations} shows each kernel as a 2D image, 
the associated TVIR and the spectrum of the operator $J$ (the linear integral operator with kernel $T$) computed with an SVD.

\begin{example}\label{example1}
 A typical kernel that motivates our study is defined by:
 \begin{equation}
  K(x,y)= \frac{1}{\sqrt{2\pi}\sigma(y)} \exp\left( - \frac{(x-y)^2}{2\sigma^2(y)}\right).
 \end{equation}
 The impulse responses $K(\cdot,y)$ are Gaussian for all $y\in \Omega$. Their variance $\sigma(y)>0$ varies depending on the position $y$.
 The TVIR of $K$ is defined by:
 \begin{equation}
  T(x,y)= \frac{1}{\sqrt{2\pi}\sigma(y)} \exp\left( - \frac{x^2}{2\sigma^2(y)}\right).
 \end{equation}
 The impulse responses $T(\cdot,y)$ are not compactly supported, therefore, $\kappa=1$ in assumption \ref{assumption2}. However, it is possible to truncate them by setting $\kappa=3\sup_{y\in \Omega} \sigma(y)$ for instance.
 This kernel satisfies assumption \ref{assumption1} only if $\sigma:\Omega \to \R$ is sufficiently smooth. 
 In figure \ref{fig:kernels_illustrations}, left column, we set $\sigma(y) = 0.08+ 0.02\cos(2\pi y)$.
\end{example}

\begin{example}\label{example2}
The second example is given by:
\begin{equation}
  T(x,y)=  \frac{2}{\sigma(y)} \max(1 - 2\sigma(y) |x|,0).
 \end{equation}
 The impulse responses $T(\cdot,y)$ are cardinal B-splines of degree $1$ and width $\sigma(y)>0$.
 They are compactly supported with $\kappa=\sup_{y\in \Omega} \sigma(y)$.
 This kernel satisfies assumption \ref{assumption2} only if $\sigma:\Omega \to \R$ is sufficiently smooth. 
 In figure \ref{fig:kernels_illustrations}, central column, we set $\sigma(y) = 0.1 + 0.3 (1-|y|)$.
 This kernel satisfies assumption \ref{assumption1} with $s=1$.
\end{example}

\begin{example}\label{example3}
The last example is a discontinuous TVIR. We set:
\begin{equation}
  T(x,y)=  g_{\sigma_1}(x)\one_{[-1/4,1/4]}(y) + g_{\sigma_2}(x)(1-\one_{[-1/4,1/4]}(y)),
 \end{equation}
 where $g_{\sigma}(x) =\frac{1}{\sqrt{2\pi}} \exp\left(-\frac{x^2}{\sigma^2}\right)$. This corresponds to the last column in figure \ref{fig:kernels_illustrations}, with $\sigma_1=0.05$ and $\sigma_2=0.1$. 
  For this kernel, both assumptions \ref{assumption1} and \ref{assumption2} are violated. Notice however that $T$ is the sum of two tensor products and can therefore be represented using only four 1D functions. 
  The spectrum of $J$ should have only 2 non zero elements. This is verified in figure \ref{fig:k33} up to numerical errors.
\end{example}

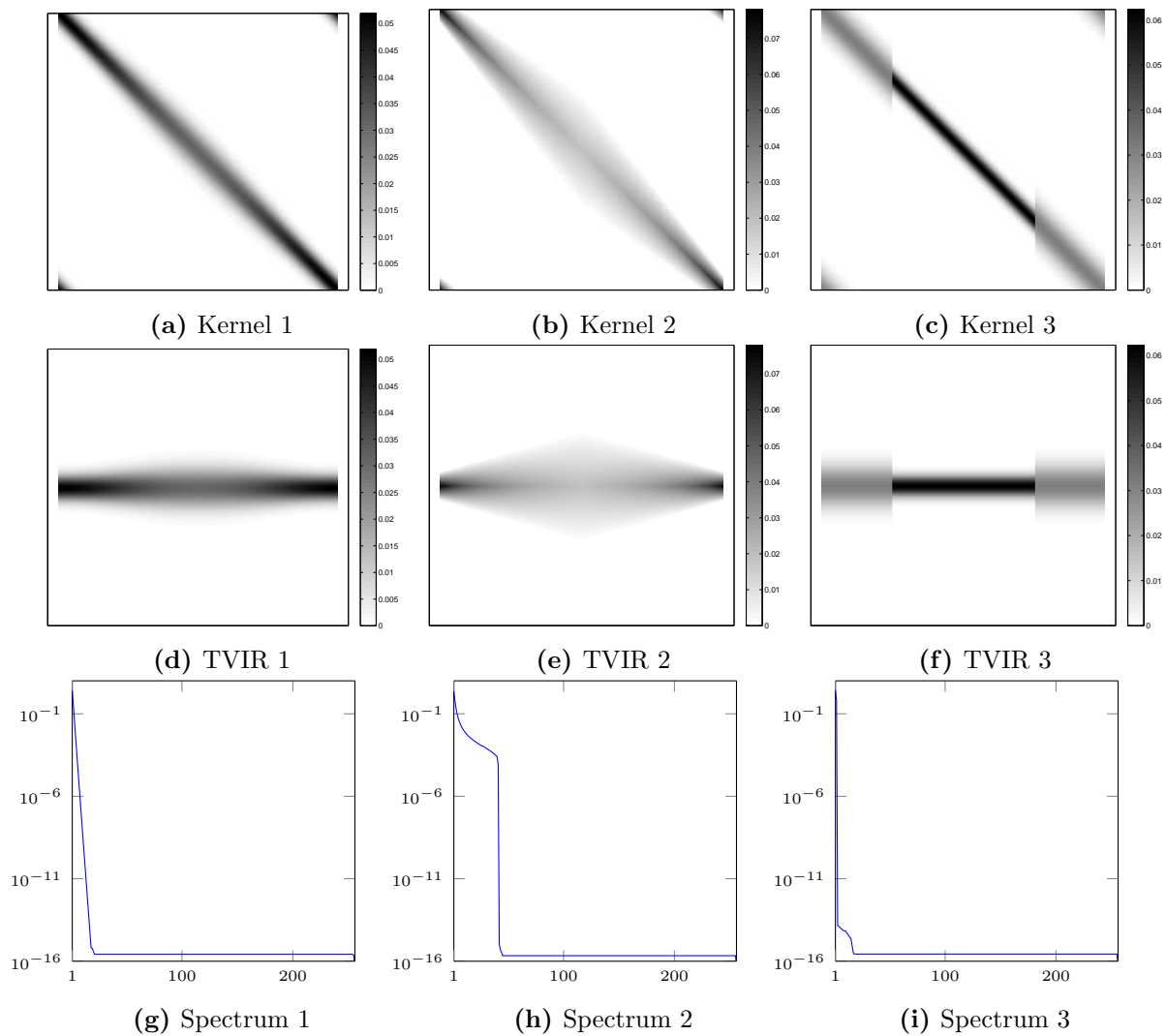
\begin{figure}
    \centering
    \begin{subfigure}[b]{0.3\textwidth}
        \includegraphics[width=\textwidth]{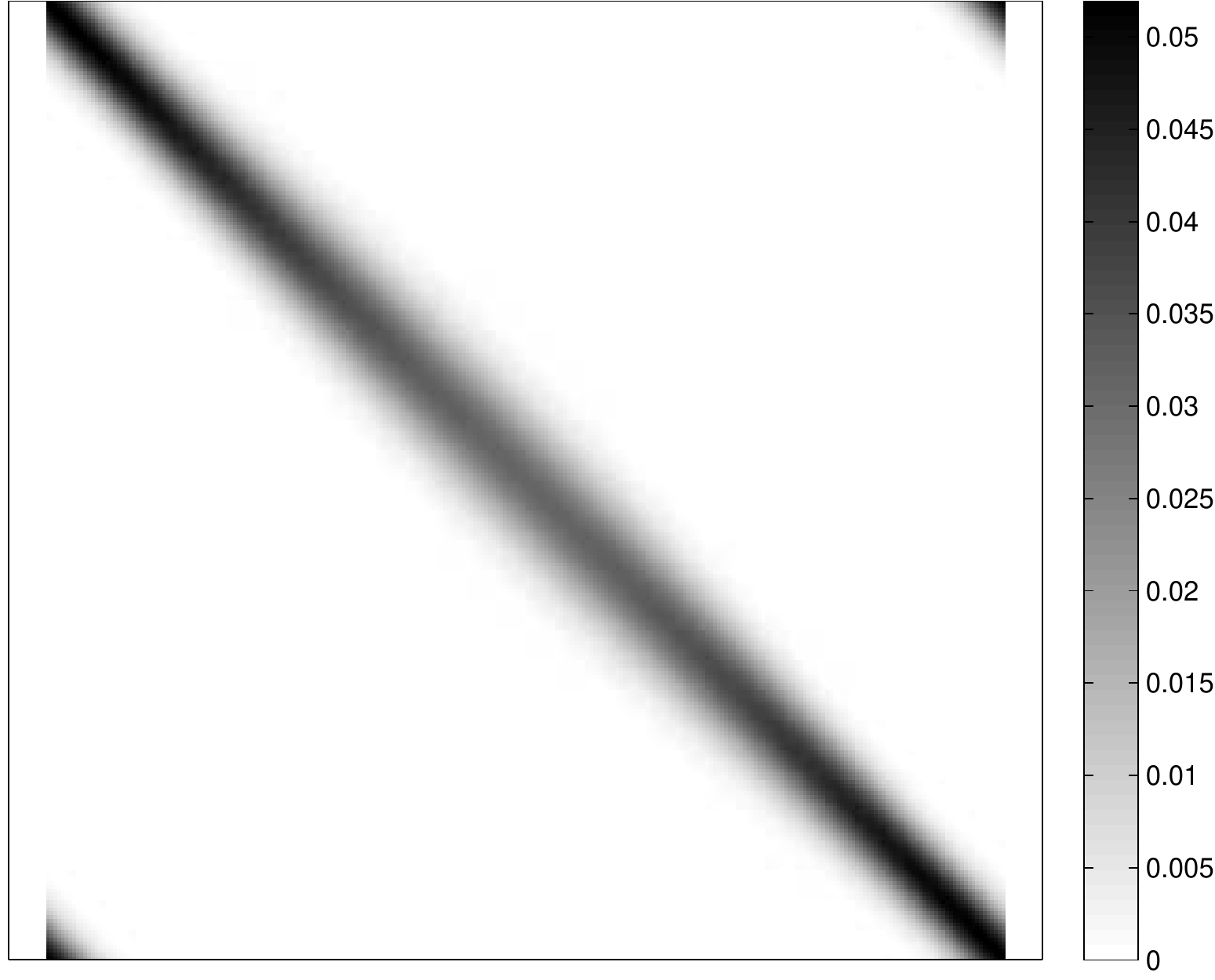}
        \caption{Kernel 1}
        \label{fig:k11}
    \end{subfigure}
    ~ 
    \begin{subfigure}[b]{0.3\textwidth}
        \includegraphics[width=\textwidth]{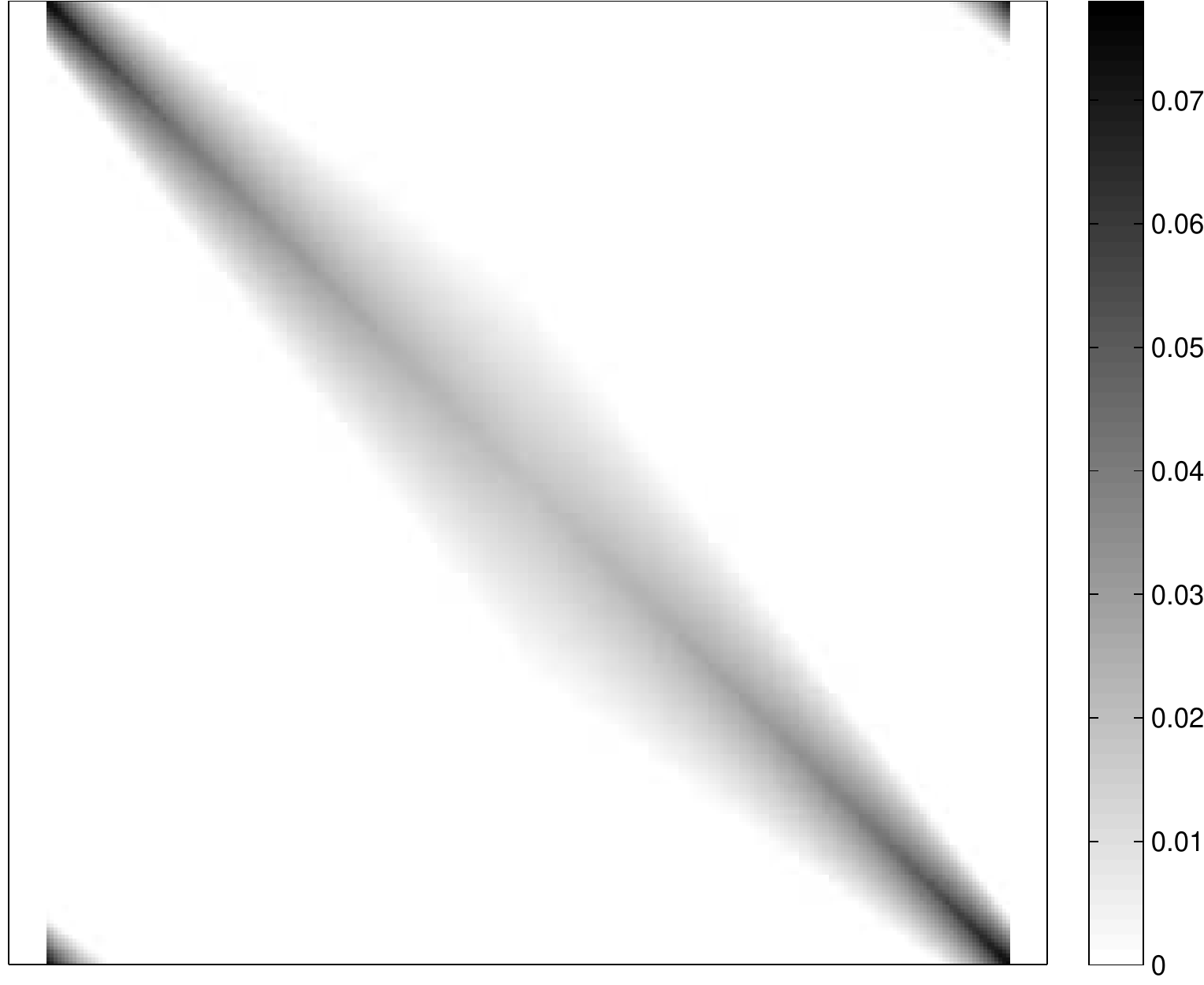}
        \caption{Kernel 2}
        \label{fig:k21}
    \end{subfigure}
    ~ 
    \begin{subfigure}[b]{0.3\textwidth}
        \includegraphics[width=\textwidth]{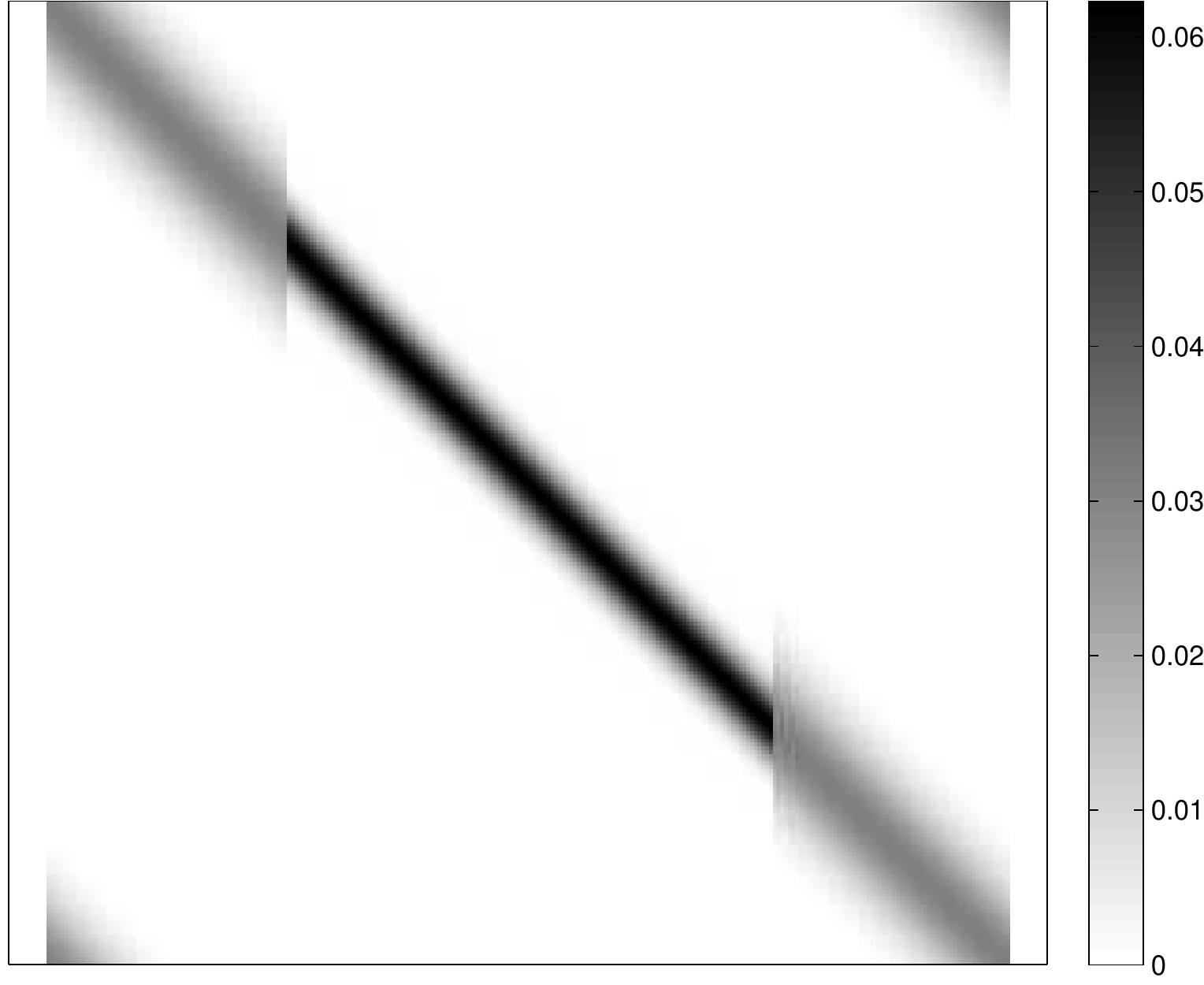}
        \caption{Kernel 3}
        \label{fig:k31}
    \end{subfigure}
    
    \begin{subfigure}[b]{0.3\textwidth}
        \includegraphics[width=\textwidth]{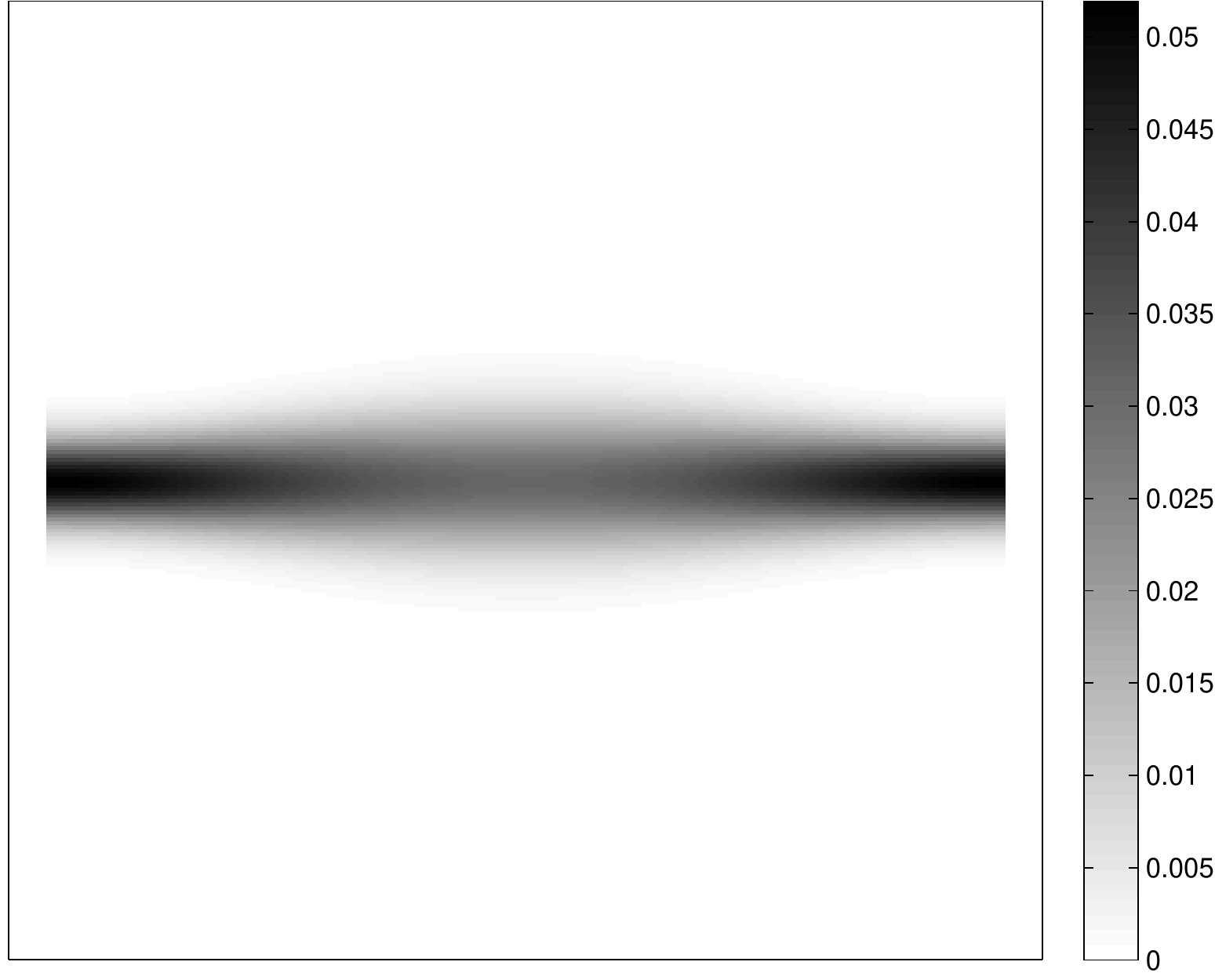}
        \caption{TVIR 1}
        \label{fig:k12}
    \end{subfigure}
    ~ 
    \begin{subfigure}[b]{0.3\textwidth}
        \includegraphics[width=\textwidth]{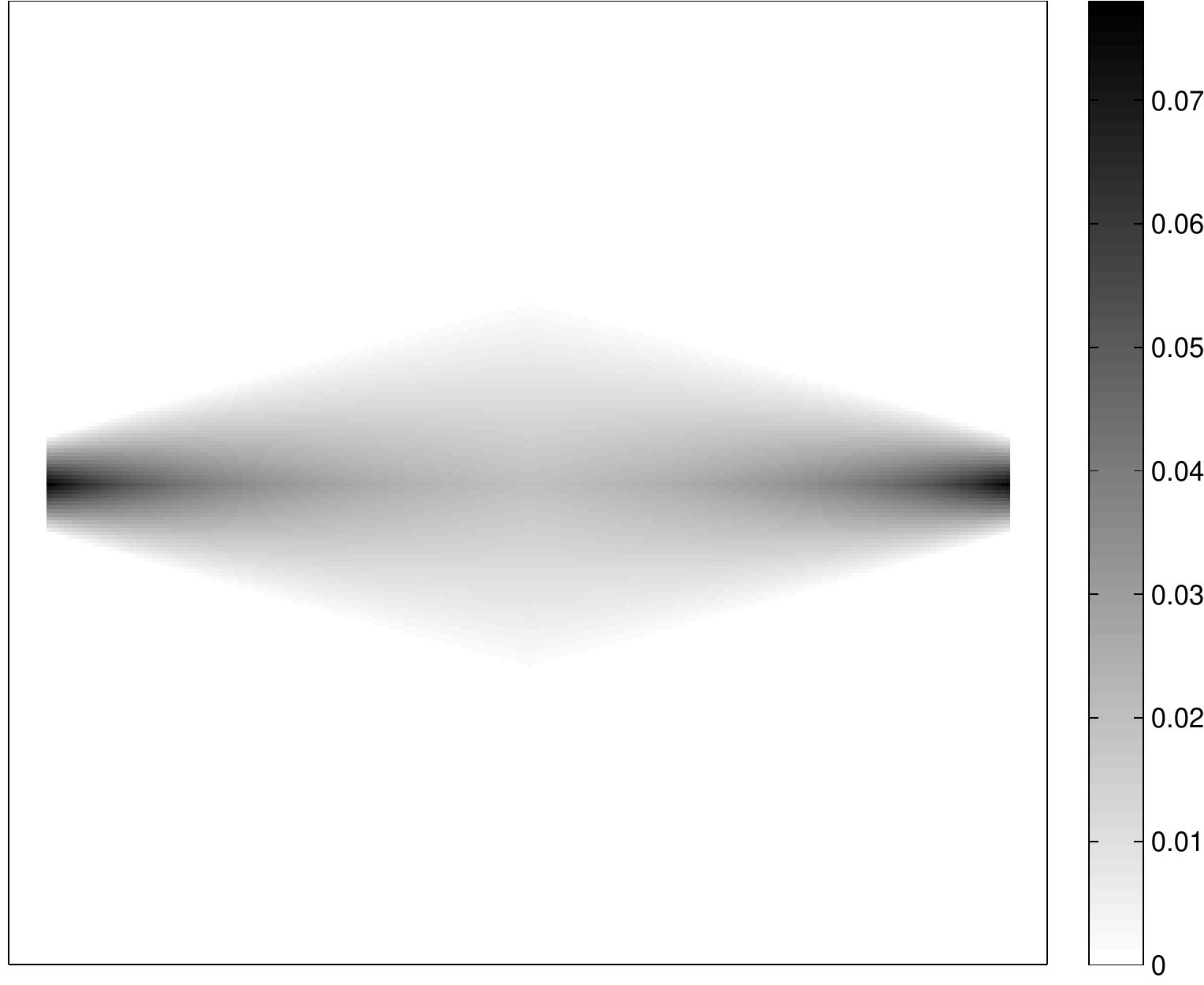}
        \caption{TVIR 2}
        \label{fig:k22}
    \end{subfigure}
    ~ 
    \begin{subfigure}[b]{0.3\textwidth}
        \includegraphics[width=\textwidth]{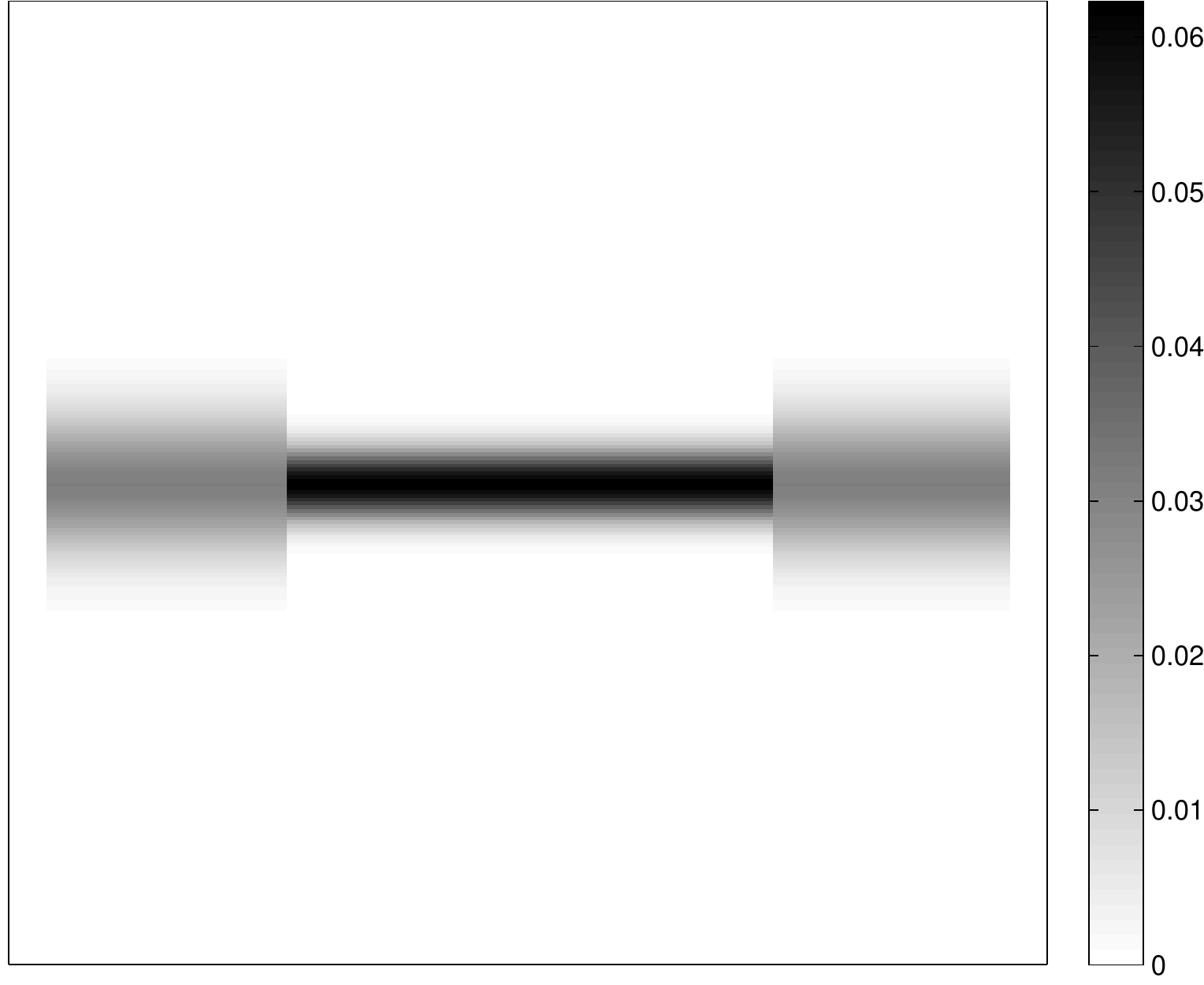}
        \caption{TVIR 3}
        \label{fig:k32}
    \end{subfigure}   
    
    \begin{subfigure}[b]{0.3\textwidth}
%
%
\begin{tikzpicture}

\hspace{-0.6cm}
\begin{axis}[%
width=0.8\textwidth,
height=0.8\textwidth,
at={(0\textwidth,0\textwidth)},
scale only axis,
separate axis lines,
every outer x axis line/.append style={black},
every x tick label/.append style={font=\tiny},
xmin=1,
xmax=256,
every outer y axis line/.append style={black},
every y tick label/.append style={font=\tiny},
ymode=log,
ymin=1e-16,
ymax=10,
yminorticks=true,
axis background/.style={fill=white},
extra x ticks ={1},
extra x tick labels={1}
]
\addplot [color=blue,solid,forget plot]
  table[row sep=crcr]{%
1	2.67780693104762\\
2	0.334183524970324\\
3	0.0361980819968495\\
4	0.00413248520979801\\
5	0.000483432623456852\\
6	5.73559970717887e-05\\
7	6.86762456681281e-06\\
8	8.27630265172482e-07\\
9	1.00217658115312e-07\\
10	1.21802094594582e-08\\
11	1.48470062188645e-09\\
12	1.81409305657456e-10\\
13	2.22095793732952e-11\\
14	2.72362316180868e-12\\
15	3.34484968765629e-13\\
16	4.11300491701563e-14\\
17	5.06550558346858e-15\\
18	6.49617955415584e-16\\
19	6.0592046480274e-16\\
20	4.10169288117018e-16\\
21	2.63304817691486e-16\\
22	2.63304817691486e-16\\
23	2.63304817691486e-16\\
24	2.63304817691486e-16\\
25	2.63304817691486e-16\\
26	2.63304817691486e-16\\
27	2.63304817691486e-16\\
28	2.63304817691486e-16\\
29	2.63304817691486e-16\\
30	2.63304817691486e-16\\
31	2.63304817691486e-16\\
32	2.63304817691486e-16\\
33	2.63304817691486e-16\\
34	2.63304817691486e-16\\
35	2.63304817691486e-16\\
36	2.63304817691486e-16\\
37	2.63304817691486e-16\\
38	2.63304817691486e-16\\
39	2.63304817691486e-16\\
40	2.63304817691486e-16\\
41	2.63304817691486e-16\\
42	2.63304817691486e-16\\
43	2.63304817691486e-16\\
44	2.63304817691486e-16\\
45	2.63304817691486e-16\\
46	2.63304817691486e-16\\
47	2.63304817691486e-16\\
48	2.63304817691486e-16\\
49	2.63304817691486e-16\\
50	2.63304817691486e-16\\
51	2.63304817691486e-16\\
52	2.63304817691486e-16\\
53	2.63304817691486e-16\\
54	2.63304817691486e-16\\
55	2.63304817691486e-16\\
56	2.63304817691486e-16\\
57	2.63304817691486e-16\\
58	2.63304817691486e-16\\
59	2.63304817691486e-16\\
60	2.63304817691486e-16\\
61	2.63304817691486e-16\\
62	2.63304817691486e-16\\
63	2.63304817691486e-16\\
64	2.63304817691486e-16\\
65	2.63304817691486e-16\\
66	2.63304817691486e-16\\
67	2.63304817691486e-16\\
68	2.63304817691486e-16\\
69	2.63304817691486e-16\\
70	2.63304817691486e-16\\
71	2.63304817691486e-16\\
72	2.63304817691486e-16\\
73	2.63304817691486e-16\\
74	2.63304817691486e-16\\
75	2.63304817691486e-16\\
76	2.63304817691486e-16\\
77	2.63304817691486e-16\\
78	2.63304817691486e-16\\
79	2.63304817691486e-16\\
80	2.63304817691486e-16\\
81	2.63304817691486e-16\\
82	2.63304817691486e-16\\
83	2.63304817691486e-16\\
84	2.63304817691486e-16\\
85	2.63304817691486e-16\\
86	2.63304817691486e-16\\
87	2.63304817691486e-16\\
88	2.63304817691486e-16\\
89	2.63304817691486e-16\\
90	2.63304817691486e-16\\
91	2.63304817691486e-16\\
92	2.63304817691486e-16\\
93	2.63304817691486e-16\\
94	2.63304817691486e-16\\
95	2.63304817691486e-16\\
96	2.63304817691486e-16\\
97	2.63304817691486e-16\\
98	2.63304817691486e-16\\
99	2.63304817691486e-16\\
100	2.63304817691486e-16\\
101	2.63304817691486e-16\\
102	2.63304817691486e-16\\
103	2.63304817691486e-16\\
104	2.63304817691486e-16\\
105	2.63304817691486e-16\\
106	2.63304817691486e-16\\
107	2.63304817691486e-16\\
108	2.63304817691486e-16\\
109	2.63304817691486e-16\\
110	2.63304817691486e-16\\
111	2.63304817691486e-16\\
112	2.63304817691486e-16\\
113	2.63304817691486e-16\\
114	2.63304817691486e-16\\
115	2.63304817691486e-16\\
116	2.63304817691486e-16\\
117	2.63304817691486e-16\\
118	2.63304817691486e-16\\
119	2.63304817691486e-16\\
120	2.63304817691486e-16\\
121	2.63304817691486e-16\\
122	2.63304817691486e-16\\
123	2.63304817691486e-16\\
124	2.63304817691486e-16\\
125	2.63304817691486e-16\\
126	2.63304817691486e-16\\
127	2.63304817691486e-16\\
128	2.63304817691486e-16\\
129	2.63304817691486e-16\\
130	2.63304817691486e-16\\
131	2.63304817691486e-16\\
132	2.63304817691486e-16\\
133	2.63304817691486e-16\\
134	2.63304817691486e-16\\
135	2.63304817691486e-16\\
136	2.63304817691486e-16\\
137	2.63304817691486e-16\\
138	2.63304817691486e-16\\
139	2.63304817691486e-16\\
140	2.63304817691486e-16\\
141	2.63304817691486e-16\\
142	2.63304817691486e-16\\
143	2.63304817691486e-16\\
144	2.63304817691486e-16\\
145	2.63304817691486e-16\\
146	2.63304817691486e-16\\
147	2.63304817691486e-16\\
148	2.63304817691486e-16\\
149	2.63304817691486e-16\\
150	2.63304817691486e-16\\
151	2.63304817691486e-16\\
152	2.63304817691486e-16\\
153	2.63304817691486e-16\\
154	2.63304817691486e-16\\
155	2.63304817691486e-16\\
156	2.63304817691486e-16\\
157	2.63304817691486e-16\\
158	2.63304817691486e-16\\
159	2.63304817691486e-16\\
160	2.63304817691486e-16\\
161	2.63304817691486e-16\\
162	2.63304817691486e-16\\
163	2.63304817691486e-16\\
164	2.63304817691486e-16\\
165	2.63304817691486e-16\\
166	2.63304817691486e-16\\
167	2.63304817691486e-16\\
168	2.63304817691486e-16\\
169	2.63304817691486e-16\\
170	2.63304817691486e-16\\
171	2.63304817691486e-16\\
172	2.63304817691486e-16\\
173	2.63304817691486e-16\\
174	2.63304817691486e-16\\
175	2.63304817691486e-16\\
176	2.63304817691486e-16\\
177	2.63304817691486e-16\\
178	2.63304817691486e-16\\
179	2.63304817691486e-16\\
180	2.63304817691486e-16\\
181	2.63304817691486e-16\\
182	2.63304817691486e-16\\
183	2.63304817691486e-16\\
184	2.63304817691486e-16\\
185	2.63304817691486e-16\\
186	2.63304817691486e-16\\
187	2.63304817691486e-16\\
188	2.63304817691486e-16\\
189	2.63304817691486e-16\\
190	2.63304817691486e-16\\
191	2.63304817691486e-16\\
192	2.63304817691486e-16\\
193	2.63304817691486e-16\\
194	2.63304817691486e-16\\
195	2.63304817691486e-16\\
196	2.63304817691486e-16\\
197	2.63304817691486e-16\\
198	2.63304817691486e-16\\
199	2.63304817691486e-16\\
200	2.63304817691486e-16\\
201	2.63304817691486e-16\\
202	2.63304817691486e-16\\
203	2.63304817691486e-16\\
204	2.63304817691486e-16\\
205	2.63304817691486e-16\\
206	2.63304817691486e-16\\
207	2.63304817691486e-16\\
208	2.63304817691486e-16\\
209	2.63304817691486e-16\\
210	2.63304817691486e-16\\
211	2.63304817691486e-16\\
212	2.63304817691486e-16\\
213	2.63304817691486e-16\\
214	2.63304817691486e-16\\
215	2.63304817691486e-16\\
216	2.63304817691486e-16\\
217	2.63304817691486e-16\\
218	2.63304817691486e-16\\
219	2.63304817691486e-16\\
220	2.63304817691486e-16\\
221	2.63304817691486e-16\\
222	2.63304817691486e-16\\
223	2.63304817691486e-16\\
224	2.63304817691486e-16\\
225	2.63304817691486e-16\\
226	2.63304817691486e-16\\
227	2.63304817691486e-16\\
228	2.63304817691486e-16\\
229	2.63304817691486e-16\\
230	2.63304817691486e-16\\
231	2.63304817691486e-16\\
232	2.63304817691486e-16\\
233	2.63304817691486e-16\\
234	2.63304817691486e-16\\
235	2.63304817691486e-16\\
236	2.63304817691486e-16\\
237	2.63304817691486e-16\\
238	2.63304817691486e-16\\
239	2.63304817691486e-16\\
240	2.63304817691486e-16\\
241	2.63304817691486e-16\\
242	2.63304817691486e-16\\
243	2.63304817691486e-16\\
244	2.63304817691486e-16\\
245	2.63304817691486e-16\\
246	2.63304817691486e-16\\
247	2.63304817691486e-16\\
248	2.63304817691486e-16\\
249	2.63304817691486e-16\\
250	2.63304817691486e-16\\
251	2.63304817691486e-16\\
252	2.63304817691486e-16\\
253	2.63304817691486e-16\\
254	2.63304817691486e-16\\
255	2.13765889953001e-16\\
256	2.43477847061029e-17\\
};
\end{axis}
\end{tikzpicture}%
	   \caption{Spectrum 1}
        \label{fig:k13}
    \end{subfigure}
    ~ 
    \begin{subfigure}[b]{0.3\textwidth}
%
%
\begin{tikzpicture}

\hspace{-0.6cm}
\begin{axis}[%
width=0.8\textwidth,
height=0.8\textwidth,
at={(0\textwidth,0\textwidth)},
scale only axis,
separate axis lines,
every outer x axis line/.append style={black},
every x tick label/.append style={font=\tiny},
xmin=1,
xmax=256,
every outer y axis line/.append style={black},
every y tick label/.append style={font=\tiny},
ymode=log,
ymin=1e-16,
ymax=10,
yminorticks=true,
axis background/.style={fill=white},
extra x ticks ={1},
extra x tick labels={1},
]
\addplot [color=blue,solid,forget plot]
  table[row sep=crcr]{%
1	2.39183344854508\\
2	0.620779375343812\\
3	0.208273597183397\\
4	0.0948261162743502\\
5	0.0525560337018792\\
6	0.033020017908125\\
7	0.0225574501587846\\
8	0.0163489444461776\\
9	0.012388235133574\\
10	0.00972250969033367\\
11	0.00782359357505018\\
12	0.00642661164021555\\
13	0.00539690380180924\\
14	0.00458542505281676\\
15	0.00395845744856025\\
16	0.00343901215293715\\
17	0.00300913063627927\\
18	0.00267040393375797\\
19	0.00239604678650024\\
20	0.00213810745401609\\
21	0.00190675362098705\\
22	0.00172451408714787\\
23	0.00155101176766781\\
24	0.00138996683836288\\
25	0.00126313664528978\\
26	0.00113293326377981\\
27	0.00111734724699201\\
28	0.00102385346577727\\
29	0.000926165801575906\\
30	0.00083130467220514\\
31	0.00075464498390089\\
32	0.000675391364797016\\
33	0.000610834973133848\\
34	0.000548130006201863\\
35	0.000487777702463827\\
36	0.000439578658398402\\
37	0.000390401494625637\\
38	0.000342312203219229\\
39	0.000296984512909132\\
40	0.000251435268093714\\
41	8.27974045249128e-05\\
42	1.06218742625734e-15\\
43	4.84064790234776e-16\\
44	3.49850237609337e-16\\
45	2.16361299065115e-16\\
46	2.16361299065115e-16\\
47	2.16361299065115e-16\\
48	2.16361299065115e-16\\
49	2.16361299065115e-16\\
50	2.16361299065115e-16\\
51	2.16361299065115e-16\\
52	2.16361299065115e-16\\
53	2.16361299065115e-16\\
54	2.16361299065115e-16\\
55	2.16361299065115e-16\\
56	2.16361299065115e-16\\
57	2.16361299065115e-16\\
58	2.16361299065115e-16\\
59	2.16361299065115e-16\\
60	2.16361299065115e-16\\
61	2.16361299065115e-16\\
62	2.16361299065115e-16\\
63	2.16361299065115e-16\\
64	2.16361299065115e-16\\
65	2.16361299065115e-16\\
66	2.16361299065115e-16\\
67	2.16361299065115e-16\\
68	2.16361299065115e-16\\
69	2.16361299065115e-16\\
70	2.16361299065115e-16\\
71	2.16361299065115e-16\\
72	2.16361299065115e-16\\
73	2.16361299065115e-16\\
74	2.16361299065115e-16\\
75	2.16361299065115e-16\\
76	2.16361299065115e-16\\
77	2.16361299065115e-16\\
78	2.16361299065115e-16\\
79	2.16361299065115e-16\\
80	2.16361299065115e-16\\
81	2.16361299065115e-16\\
82	2.16361299065115e-16\\
83	2.16361299065115e-16\\
84	2.16361299065115e-16\\
85	2.16361299065115e-16\\
86	2.16361299065115e-16\\
87	2.16361299065115e-16\\
88	2.16361299065115e-16\\
89	2.16361299065115e-16\\
90	2.16361299065115e-16\\
91	2.16361299065115e-16\\
92	2.16361299065115e-16\\
93	2.16361299065115e-16\\
94	2.16361299065115e-16\\
95	2.16361299065115e-16\\
96	2.16361299065115e-16\\
97	2.16361299065115e-16\\
98	2.16361299065115e-16\\
99	2.16361299065115e-16\\
100	2.16361299065115e-16\\
101	2.16361299065115e-16\\
102	2.16361299065115e-16\\
103	2.16361299065115e-16\\
104	2.16361299065115e-16\\
105	2.16361299065115e-16\\
106	2.16361299065115e-16\\
107	2.16361299065115e-16\\
108	2.16361299065115e-16\\
109	2.16361299065115e-16\\
110	2.16361299065115e-16\\
111	2.16361299065115e-16\\
112	2.16361299065115e-16\\
113	2.16361299065115e-16\\
114	2.16361299065115e-16\\
115	2.16361299065115e-16\\
116	2.16361299065115e-16\\
117	2.16361299065115e-16\\
118	2.16361299065115e-16\\
119	2.16361299065115e-16\\
120	2.16361299065115e-16\\
121	2.16361299065115e-16\\
122	2.16361299065115e-16\\
123	2.16361299065115e-16\\
124	2.16361299065115e-16\\
125	2.16361299065115e-16\\
126	2.16361299065115e-16\\
127	2.16361299065115e-16\\
128	2.16361299065115e-16\\
129	2.16361299065115e-16\\
130	2.16361299065115e-16\\
131	2.16361299065115e-16\\
132	2.16361299065115e-16\\
133	2.16361299065115e-16\\
134	2.16361299065115e-16\\
135	2.16361299065115e-16\\
136	2.16361299065115e-16\\
137	2.16361299065115e-16\\
138	2.16361299065115e-16\\
139	2.16361299065115e-16\\
140	2.16361299065115e-16\\
141	2.16361299065115e-16\\
142	2.16361299065115e-16\\
143	2.16361299065115e-16\\
144	2.16361299065115e-16\\
145	2.16361299065115e-16\\
146	2.16361299065115e-16\\
147	2.16361299065115e-16\\
148	2.16361299065115e-16\\
149	2.16361299065115e-16\\
150	2.16361299065115e-16\\
151	2.16361299065115e-16\\
152	2.16361299065115e-16\\
153	2.16361299065115e-16\\
154	2.16361299065115e-16\\
155	2.16361299065115e-16\\
156	2.16361299065115e-16\\
157	2.16361299065115e-16\\
158	2.16361299065115e-16\\
159	2.16361299065115e-16\\
160	2.16361299065115e-16\\
161	2.16361299065115e-16\\
162	2.16361299065115e-16\\
163	2.16361299065115e-16\\
164	2.16361299065115e-16\\
165	2.16361299065115e-16\\
166	2.16361299065115e-16\\
167	2.16361299065115e-16\\
168	2.16361299065115e-16\\
169	2.16361299065115e-16\\
170	2.16361299065115e-16\\
171	2.16361299065115e-16\\
172	2.16361299065115e-16\\
173	2.16361299065115e-16\\
174	2.16361299065115e-16\\
175	2.16361299065115e-16\\
176	2.16361299065115e-16\\
177	2.16361299065115e-16\\
178	2.16361299065115e-16\\
179	2.16361299065115e-16\\
180	2.16361299065115e-16\\
181	2.16361299065115e-16\\
182	2.16361299065115e-16\\
183	2.16361299065115e-16\\
184	2.16361299065115e-16\\
185	2.16361299065115e-16\\
186	2.16361299065115e-16\\
187	2.16361299065115e-16\\
188	2.16361299065115e-16\\
189	2.16361299065115e-16\\
190	2.16361299065115e-16\\
191	2.16361299065115e-16\\
192	2.16361299065115e-16\\
193	2.16361299065115e-16\\
194	2.16361299065115e-16\\
195	2.16361299065115e-16\\
196	2.16361299065115e-16\\
197	2.16361299065115e-16\\
198	2.16361299065115e-16\\
199	2.16361299065115e-16\\
200	2.16361299065115e-16\\
201	2.16361299065115e-16\\
202	2.16361299065115e-16\\
203	2.16361299065115e-16\\
204	2.16361299065115e-16\\
205	2.16361299065115e-16\\
206	2.16361299065115e-16\\
207	2.16361299065115e-16\\
208	2.16361299065115e-16\\
209	2.16361299065115e-16\\
210	2.16361299065115e-16\\
211	2.16361299065115e-16\\
212	2.16361299065115e-16\\
213	2.16361299065115e-16\\
214	2.16361299065115e-16\\
215	2.16361299065115e-16\\
216	2.16361299065115e-16\\
217	2.16361299065115e-16\\
218	2.16361299065115e-16\\
219	2.16361299065115e-16\\
220	2.16361299065115e-16\\
221	2.16361299065115e-16\\
222	2.16361299065115e-16\\
223	2.16361299065115e-16\\
224	2.16361299065115e-16\\
225	2.16361299065115e-16\\
226	2.16361299065115e-16\\
227	2.16361299065115e-16\\
228	2.16361299065115e-16\\
229	2.16361299065115e-16\\
230	2.16361299065115e-16\\
231	2.16361299065115e-16\\
232	2.16361299065115e-16\\
233	2.16361299065115e-16\\
234	2.16361299065115e-16\\
235	2.16361299065115e-16\\
236	2.16361299065115e-16\\
237	2.16361299065115e-16\\
238	2.16361299065115e-16\\
239	2.16361299065115e-16\\
240	2.16361299065115e-16\\
241	2.16361299065115e-16\\
242	2.16361299065115e-16\\
243	2.16361299065115e-16\\
244	2.16361299065115e-16\\
245	2.16361299065115e-16\\
246	2.16361299065115e-16\\
247	2.16361299065115e-16\\
248	2.16361299065115e-16\\
249	2.16361299065115e-16\\
250	2.16361299065115e-16\\
251	2.16361299065115e-16\\
252	2.16361299065115e-16\\
253	2.16361299065115e-16\\
254	2.16361299065115e-16\\
255	2.16217570140592e-16\\
256	7.64536669508705e-17\\
};
\end{axis}
\end{tikzpicture}%
	   \caption{Spectrum 2}
        \label{fig:k23}
    \end{subfigure}
    ~ 
    \begin{subfigure}[b]{0.3\textwidth}
%
%
\begin{tikzpicture}

\hspace{-0.6cm}
\begin{axis}[%
width=0.8\textwidth,
height=0.8\textwidth,
at={(0\textwidth,0\textwidth)},
scale only axis,
separate axis lines,
every outer x axis line/.append style={black},
every x tick label/.append style={font=\tiny},
xmin=1,
xmax=256,
every outer y axis line/.append style={black},
every y tick label/.append style={font=\tiny},
ymode=log,
ymin=1e-16,
ymax=10,
yminorticks=true,
axis background/.style={fill=white},
extra x ticks ={1},
extra x tick labels={1}
]
\addplot [color=blue,solid,forget plot]
  table[row sep=crcr]{%
1	2.84455961068266\\
2	0.627186599104773\\
3	1.42925953153356e-14\\
4	1.2967543550273e-14\\
5	1.12817599513288e-14\\
6	9.44855515548304e-15\\
7	8.77644118106567e-15\\
8	7.23214801795063e-15\\
9	7.07369274966177e-15\\
10	6.87933549515747e-15\\
11	5.47828482256744e-15\\
12	4.57942209325791e-15\\
13	3.46795445452456e-15\\
14	2.81000859274028e-15\\
15	2.3548801820767e-15\\
16	9.24214486558748e-16\\
17	3.41616895911465e-16\\
18	2.72381220395581e-16\\
19	2.70443031817326e-16\\
20	2.70443031817326e-16\\
21	2.70443031817326e-16\\
22	2.70443031817326e-16\\
23	2.70443031817326e-16\\
24	2.70443031817326e-16\\
25	2.70443031817326e-16\\
26	2.70443031817326e-16\\
27	2.70443031817326e-16\\
28	2.70443031817326e-16\\
29	2.70443031817326e-16\\
30	2.70443031817326e-16\\
31	2.70443031817326e-16\\
32	2.70443031817326e-16\\
33	2.70443031817326e-16\\
34	2.70443031817326e-16\\
35	2.70443031817326e-16\\
36	2.70443031817326e-16\\
37	2.70443031817326e-16\\
38	2.70443031817326e-16\\
39	2.70443031817326e-16\\
40	2.70443031817326e-16\\
41	2.70443031817326e-16\\
42	2.70443031817326e-16\\
43	2.70443031817326e-16\\
44	2.70443031817326e-16\\
45	2.70443031817326e-16\\
46	2.70443031817326e-16\\
47	2.70443031817326e-16\\
48	2.70443031817326e-16\\
49	2.70443031817326e-16\\
50	2.70443031817326e-16\\
51	2.70443031817326e-16\\
52	2.70443031817326e-16\\
53	2.70443031817326e-16\\
54	2.70443031817326e-16\\
55	2.70443031817326e-16\\
56	2.70443031817326e-16\\
57	2.70443031817326e-16\\
58	2.70443031817326e-16\\
59	2.70443031817326e-16\\
60	2.70443031817326e-16\\
61	2.70443031817326e-16\\
62	2.70443031817326e-16\\
63	2.70443031817326e-16\\
64	2.70443031817326e-16\\
65	2.70443031817326e-16\\
66	2.70443031817326e-16\\
67	2.70443031817326e-16\\
68	2.70443031817326e-16\\
69	2.70443031817326e-16\\
70	2.70443031817326e-16\\
71	2.70443031817326e-16\\
72	2.70443031817326e-16\\
73	2.70443031817326e-16\\
74	2.70443031817326e-16\\
75	2.70443031817326e-16\\
76	2.70443031817326e-16\\
77	2.70443031817326e-16\\
78	2.70443031817326e-16\\
79	2.70443031817326e-16\\
80	2.70443031817326e-16\\
81	2.70443031817326e-16\\
82	2.70443031817326e-16\\
83	2.70443031817326e-16\\
84	2.70443031817326e-16\\
85	2.70443031817326e-16\\
86	2.70443031817326e-16\\
87	2.70443031817326e-16\\
88	2.70443031817326e-16\\
89	2.70443031817326e-16\\
90	2.70443031817326e-16\\
91	2.70443031817326e-16\\
92	2.70443031817326e-16\\
93	2.70443031817326e-16\\
94	2.70443031817326e-16\\
95	2.70443031817326e-16\\
96	2.70443031817326e-16\\
97	2.70443031817326e-16\\
98	2.70443031817326e-16\\
99	2.70443031817326e-16\\
100	2.70443031817326e-16\\
101	2.70443031817326e-16\\
102	2.70443031817326e-16\\
103	2.70443031817326e-16\\
104	2.70443031817326e-16\\
105	2.70443031817326e-16\\
106	2.70443031817326e-16\\
107	2.70443031817326e-16\\
108	2.70443031817326e-16\\
109	2.70443031817326e-16\\
110	2.70443031817326e-16\\
111	2.70443031817326e-16\\
112	2.70443031817326e-16\\
113	2.70443031817326e-16\\
114	2.70443031817326e-16\\
115	2.70443031817326e-16\\
116	2.70443031817326e-16\\
117	2.70443031817326e-16\\
118	2.70443031817326e-16\\
119	2.70443031817326e-16\\
120	2.70443031817326e-16\\
121	2.70443031817326e-16\\
122	2.70443031817326e-16\\
123	2.70443031817326e-16\\
124	2.70443031817326e-16\\
125	2.70443031817326e-16\\
126	2.70443031817326e-16\\
127	2.70443031817326e-16\\
128	2.70443031817326e-16\\
129	2.70443031817326e-16\\
130	2.70443031817326e-16\\
131	2.70443031817326e-16\\
132	2.70443031817326e-16\\
133	2.70443031817326e-16\\
134	2.70443031817326e-16\\
135	2.70443031817326e-16\\
136	2.70443031817326e-16\\
137	2.70443031817326e-16\\
138	2.70443031817326e-16\\
139	2.70443031817326e-16\\
140	2.70443031817326e-16\\
141	2.70443031817326e-16\\
142	2.70443031817326e-16\\
143	2.70443031817326e-16\\
144	2.70443031817326e-16\\
145	2.70443031817326e-16\\
146	2.70443031817326e-16\\
147	2.70443031817326e-16\\
148	2.70443031817326e-16\\
149	2.70443031817326e-16\\
150	2.70443031817326e-16\\
151	2.70443031817326e-16\\
152	2.70443031817326e-16\\
153	2.70443031817326e-16\\
154	2.70443031817326e-16\\
155	2.70443031817326e-16\\
156	2.70443031817326e-16\\
157	2.70443031817326e-16\\
158	2.70443031817326e-16\\
159	2.70443031817326e-16\\
160	2.70443031817326e-16\\
161	2.70443031817326e-16\\
162	2.70443031817326e-16\\
163	2.70443031817326e-16\\
164	2.70443031817326e-16\\
165	2.70443031817326e-16\\
166	2.70443031817326e-16\\
167	2.70443031817326e-16\\
168	2.70443031817326e-16\\
169	2.70443031817326e-16\\
170	2.70443031817326e-16\\
171	2.70443031817326e-16\\
172	2.70443031817326e-16\\
173	2.70443031817326e-16\\
174	2.70443031817326e-16\\
175	2.70443031817326e-16\\
176	2.70443031817326e-16\\
177	2.70443031817326e-16\\
178	2.70443031817326e-16\\
179	2.70443031817326e-16\\
180	2.70443031817326e-16\\
181	2.70443031817326e-16\\
182	2.70443031817326e-16\\
183	2.70443031817326e-16\\
184	2.70443031817326e-16\\
185	2.70443031817326e-16\\
186	2.70443031817326e-16\\
187	2.70443031817326e-16\\
188	2.70443031817326e-16\\
189	2.70443031817326e-16\\
190	2.70443031817326e-16\\
191	2.70443031817326e-16\\
192	2.70443031817326e-16\\
193	2.70443031817326e-16\\
194	2.70443031817326e-16\\
195	2.70443031817326e-16\\
196	2.70443031817326e-16\\
197	2.70443031817326e-16\\
198	2.70443031817326e-16\\
199	2.70443031817326e-16\\
200	2.70443031817326e-16\\
201	2.70443031817326e-16\\
202	2.70443031817326e-16\\
203	2.70443031817326e-16\\
204	2.70443031817326e-16\\
205	2.70443031817326e-16\\
206	2.70443031817326e-16\\
207	2.70443031817326e-16\\
208	2.70443031817326e-16\\
209	2.70443031817326e-16\\
210	2.70443031817326e-16\\
211	2.70443031817326e-16\\
212	2.70443031817326e-16\\
213	2.70443031817326e-16\\
214	2.70443031817326e-16\\
215	2.70443031817326e-16\\
216	2.70443031817326e-16\\
217	2.70443031817326e-16\\
218	2.70443031817326e-16\\
219	2.70443031817326e-16\\
220	2.70443031817326e-16\\
221	2.70443031817326e-16\\
222	2.70443031817326e-16\\
223	2.70443031817326e-16\\
224	2.70443031817326e-16\\
225	2.70443031817326e-16\\
226	2.70443031817326e-16\\
227	2.70443031817326e-16\\
228	2.70443031817326e-16\\
229	2.70443031817326e-16\\
230	2.70443031817326e-16\\
231	2.70443031817326e-16\\
232	2.70443031817326e-16\\
233	2.70443031817326e-16\\
234	2.70443031817326e-16\\
235	2.70443031817326e-16\\
236	2.70443031817326e-16\\
237	2.70443031817326e-16\\
238	2.70443031817326e-16\\
239	2.70443031817326e-16\\
240	2.70443031817326e-16\\
241	2.70443031817326e-16\\
242	2.70443031817326e-16\\
243	2.70443031817326e-16\\
244	2.70443031817326e-16\\
245	2.70443031817326e-16\\
246	2.70443031817326e-16\\
247	2.70443031817326e-16\\
248	2.70443031817326e-16\\
249	2.70443031817326e-16\\
250	2.70443031817326e-16\\
251	2.70443031817326e-16\\
252	2.70443031817326e-16\\
253	2.70443031817326e-16\\
254	2.70443031817326e-16\\
255	2.70443031817326e-16\\
256	5.75335761813221e-17\\
};
\end{axis}
\end{tikzpicture}%
	   \caption{Spectrum 3}
	   \label{fig:k33}
    \end{subfigure}
    \caption{Different kernels $K$, the associated TVIR $T$ and the spectrum of the operator $J$. Left column corresponds to example \ref{example1}. Central column corresponds to example \ref{example2}. Right column corresponds to example \ref{example3}.} \label{fig:kernels_illustrations} 
\end{figure}

\subsection{Convolution-product expansions as low-rank approximations} \label{sec:interpretation}

Though similar in spirit, product-convolution \eqref{eq:productconvolution} and convolution-product \eqref{eq:convolutionproduct} expansions have a quite different interpretation captured by the following lemma. 
\begin{lemma}
The TVIR of the product-convolution expansion $T_m$ in \eqref{eq:productconvolution} is given by:
\begin{equation}\label{eq:whatisthat}
 T_m(x,y)= \sum_{k=1}^m h_k(x) w_k(x+y).
\end{equation}
The TVIR of the convolution-product expansion $T_m$ in \eqref{eq:convolutionproduct} is given by:
\begin{equation}\label{eq:lowrank}
 T_m(x,y)= \sum_{k=1}^m h_k(x) w_k(y).
\end{equation}
\end{lemma}
\begin{proof}
We only prove \eqref{eq:lowrank} since the proof of \eqref{eq:whatisthat} relies on the same arguments. By definition:
  \begin{align}
  (H_m u)(x) &= \left(\sum_{k=1}^m h_k \star (w_k \odot u)\right)(x) \\
  &= \int_{\Omega} \sum_{k=1}^m  h_k(x-y) w_k(y) u(y) \,dy.  
  \end{align}
  By identification, this yields:
  \begin{equation}
   K_m(x,y) = \sum_{k=1}^m h_k(x-y) w_k(y),
  \end{equation}
  so that 
  \begin{equation}
   T_m(x,y) = \sum_{k=1}^m h_k(x) w_k(y).
  \end{equation}
\end{proof}

As can be seen in \eqref{eq:lowrank}, convolution-product expansions consist of finding low-rank approximations of the TVIR. 
This interpretation was already proposed in \cite{denis2015fast} for instance and is the key observation to derive the forthcoming results. The expansion \eqref{eq:whatisthat} does not share this simple interpretation and we do not investigate it further in this paper. 

\subsection{Discretization} \label{sec:discretization}

In order to implement a convolution-product expansion of type \ref{eq:convolutionproduct}, the problem first needs to be discretized. Discretization is a hard problem in itself and we treat it superficially in this paper with a Galerkin formalism. 
Let $(\varphi_1,\hdots, \varphi_n)$ be a basis of a finite dimensional vector space $V^n$ of $L^2(\Omega)$. Given an operator $H:L^2(\Omega)\to L^2(\Omega)$, we can construct a matrix $\b{H}^n\in \R^{n\times n}$ defined for all $1\leq i,j\leq n$ by $\b{H}^n[i,j] = \langle H\varphi_j, \varphi_i \rangle.$
Let $S^n:H\mapsto \b{H}^n$ denote the discretization operator. From a matrix $\b{H}^n$, an operator $H^n$ can be reconstructed using, for instance, the pseudo-inverse $S^{n,+}$ of $S^n$. We let $H^n=S^{n,+}(\b{H}^n)$.
For instance, if $(\varphi_1,\hdots, \varphi_n)$ is an \emph{orthonormal} basis of $V^n$, the operator $H^n$ is given by:
\begin{equation}
 H^n = S^{n,+}(\b{H}^n) = \sum_{1\leq i,j\leq n} \b{H}^n[i,j] \varphi_i\otimes \varphi_j.
\end{equation}

This paper is dedicated to analyzing methods denoted $\mathcal{A}_m$ that provide an approximation $H_m=\mathcal{A}_m(H)$ of type \eqref{eq:convolutionproduct}, given an input operator $H$. Our analysis provides guarantees on the distance $\|H-H_m\|_{HS}$ depending on $m$ and the regularity properties of the input operator $H$, for different methods. 
Depending on the context, two different approaches can be used to implement $\mathcal{A}_m$.
\begin{itemize}
 \item Compute the matrix $\b{H}_m^n = S^n(H_m)$ using numerical integration procedures. Then create an operator $H_m^n=S^{n,+}(\b{H}_m^n)$. This approach suffers from two defects. First, it is only possible by assuming that the kernel of $H$ is given analytically. Moreover it might be computationally intractable. It is illustrated below.
\begin{center}
\begin{picture}(350,30)
\put(0,0){\framebox(50,20){$H$}}
\put(50,10){\vector(1,0){40}}
\put(60,15) {$\mathcal{A}_m$}
\put(90,0){\framebox(50,20){$H_m$}}
\put(140,10){\vector(1,0){40}}
\put(150,15) {$S^n$}
\put(180,0){\framebox(50,20){$\b{H}_m^n$}}
\put(230,10){\vector(1,0){40}}
\put(240,15) {$S^{n,+}$}
\put(270,0){\framebox(50,20){$H_m^n$}}
\end{picture}
\end{center}

 \item In many applications, the operator $H$ is not given explicitly. Instead, we only have access to its discretization $\b{H}^n$. Then it is possible to construct a discrete approximation algorithm $\b{\mathcal{A}}_m$ yielding a discrete approximation $\b{H}_m^n = \b{\mathcal{A}}_m(\b{H}^n)$. This matrix can then be mapped back to the continuous world using the pseudo-inverse: $H_m^n=S^{n,+}(\b{H}_m^n)$. This is illustrated below. In this paper, we will analyze the construction complexity of $\b{H}_m^n$ using this second approach.
\begin{center}
\begin{picture}(350,30)
\put(0,0){\framebox(50,20){$H$}}
\put(50,10){\vector(1,0){40}}
\put(60,15) {$S^n$}
\put(90,0){\framebox(50,20){$\b{H}^n$}}
\put(140,10){\vector(1,0){40}}
\put(150,15) {$\b{\mathcal{A}}_m$}
\put(180,0){\framebox(50,20){$\b{H}_m^n$}}
\put(230,10){\vector(1,0){40}}
\put(240,15) {$S^{n,+}$}
\put(270,0){\framebox(50,20){$H_m^n$}}
\end{picture}
\end{center} 
\end{itemize}

Ideally, we would like to provide guarantees on $\|H-H_m^n\|_{HS}$ depending on $m$ and $n$. 
In the first approach, this is possible by using the following inequality:
\begin{equation}
 \|H-H_m^n\|_{HS} \leq  \underbrace{\|H - H_m\|_{HS}}_{\epsilon_a(m)} +  \underbrace{\|H_m - H_m^n\|_{HS}}_{\epsilon_d(n)},
\end{equation}
where $\epsilon_a(m)$ is the approximation error studied in this paper and $\epsilon_d(n)$ is the discretization error. 

In the second approach, the error analysis is more complex since there is an additional bias due to the  algorithm discretization.
This bias is captured by the following inequality:
\begin{equation}
 \|H-H_m^n\|_{HS} \leq  \underbrace{\|H - H^n\|_{HS}}_{\epsilon_d(n)} +  \underbrace{\|H^n - \mathcal{A}_m(H^n) \|_{HS}}_{\epsilon_a(m)} + \underbrace{\|\mathcal{A}_m(H^n) - H_m^n\|_{HS}}_{\epsilon_b(m,n)}.
\end{equation}
The bias $\epsilon_b(m,n) = \|\mathcal{A}_m(S^{n,+}(S^n(H))) - S^{n,+}( \b{\mathcal{A}}_m( S^n(H))) \|_{HS}$ accounts for the difference between using the discrete or continuous approximation algorithm.

\emph{In all the paper, we assume - without mention - that $\epsilon_d(n)$ and $\epsilon_b(m,n)$ are negligible compared to $\epsilon_a(m)$.}

%
%

\subsection{Implementation and complexity}

Let $\b{F}_n\in \C^{n\times n}$ denote the discrete inverse Fourier transform and $\b{F}_n^*$ denote the discrete Fourier transform.
Matrix-vector products $\b{F}_n\b{u}$ or $\b{F}_n^*\b{u}$ can be evaluated in $O(n\log_2(n))$ operations using the fast Fourier transform (FFT). 
The discrete convolution product $\b{v}=\b{h}\star \b{u}$ is defined for all $i\in \Z$ by $\b{v}[i] = \sum_{j=1}^n \b{u}[i-j]\b{h}[j]$, with circular boundary conditions. 

Discrete convolution products can be evaluated in $O(n\log_2(n))$ operations by using the following fundamental identity:
\begin{equation}
 \b{v} = \b{F}_n \cdot ( (\b{F}_n^*\b{h}) \odot (\b{F}_n^*\b{u})).
\end{equation}
Hence a convolution can be implemented using three FFTs ($O(n\log_2(n))$ operations) and a point-wise multiplication ($O(n)$ operations).
This being said, it is straightforward to implement formula \eqref{eq:convolutionproduct} with an $O(mn\log_2(n))$ algorithm. 

Under the additional assumption that $w_k$ and $h_k$ are supported on bounded intervals, the complexity can be improved.
We assume that, after discretization, $\b{h}_k$ and $\b{w}_k$ are compactly supported, with support length $q_k\leq n$ and $p_k\leq n$ respectively.
\begin{lemma} \label{lem:complexity}
 A matrix-vector product of type \eqref{eq:convolutionproduct} can be implemented with a complexity that does not exceed $O\left( \sum_{k=1}^m (p_k+q_k) \log_2(\min(p_k,q_k)) \right)$ operations.
\end{lemma}
\begin{proof}
A convolution product of type $\b{h}_k\star (\b{w}_k\odot \b{u})$ can be evaluated in $O((p_k+q_k) \log(p_k+q_k))$ operations. Indeed, the support of $\b{h}_k\star (\b{w}_k\odot \b{u})$ has no more than $p_k+q_k$ contiguous non-zeros elements. Using the Stockham sectioning algorithm \cite{stockham1966high}, the complexity can be further decreased to $O((p_k+q_k) \log_2(\min(p_k,q_k)))$ operations. This idea was proposed in \cite{hirsch2010efficient}.
\end{proof}
\section{Projections on linear subspaces}\label{sec:linearsubspace}

We now turn to the problem of choosing the functions $h_k$ and $w_k$ in equation \eqref{eq:convolutionproduct}. The idea studied in this section is to fix a subspace $E_m=\mathrm{span}(e_k, k\in \{1,\hdots, m\})$ of $L^2(\Omega)$ and to approximate $T(x,\cdot)$ as:
\begin{equation}\label{eq:decomposition_linear_subspace}
 T_m(x,y) = \sum_{k=1}^m c_k(x) e_k(y).
\end{equation}
For instance, the coefficients $c_k$ can be chosen so that $T_m(x,\cdot)$ is a projection of $T(x,\cdot)$ onto $E_m$.
We propose to analyze three different different family of functions $e_k$: Fourier atoms, wavelets atoms and B-splines. 
We analyze their complexity and approximation properties as well as their respective advantages.

\subsection{Fourier decompositions}\label{sec:Fourier}

It is well known that functions in  $H^s(\Omega)$ can be well approximated by linear combination of low-frequency Fourier atoms. 
This loose statement is captured by the following lemma.
\begin{lemma}[{\cite{devore1993constructive,devore1998nonlinear}}]\label{lem:approx_rate_Fourier}
 Let $f\in H^s(\Omega)$ and $f_m$ denote its partial Fourier series:
 \begin{equation}
  f_m=\sum_{k=-m}^m \hat f[k] e_k,
 \end{equation}
 where $e_k(y) =\exp(-2 i \pi k y)$. Then 
 \begin{equation}
  \|f_m-f\|_{L^2(\Omega)} \leq C m^{-s} |f|_{H^s(\Omega)}.
 \end{equation}
\end{lemma}

The so-called Kohn-Nirenberg symbol $N$ of $H$ is defined for all $(x,k)\in \Omega\times \Z$ by
\begin{equation}
 N(x,k) = \int_{\Omega} T(x,y) \exp(-2i\pi ky) \,dy.
\end{equation} 
Illustrations of different Kohn-Nirenberg symbols are provided in figure \ref{fig:Kohn_Nirenberg}.

\begin{corollary}\label{thm:approx_Fourier}
Set $e_k(y) =\exp(-2 i \pi k y)$ and define $T_m$ by:
\begin{equation} \label{eq:Fourier_form}
 T_m(x,y)=\sum_{|k|\leq m} N(x,k) e_k(y).
\end{equation}
Then, under assumptions \ref{assumption1} and \ref{assumption2}
\begin{equation}\label{eq:Fourier_rate}
\|H_m - H\|_{HS}  \leq C \sqrt{\kappa}m^{-s}.
\end{equation}
\end{corollary}
\begin{proof}
By lemma \ref{lem:approx_rate_Fourier} and assumption \ref{assumption1}, $\|T_m(x,\cdot)-T(x,\cdot) \|_{L^2(\Omega)} \leq C m^{-s}$ for some constant $C$ and for all $x \in \Omega$. In addition, by assumption \ref{assumption2}, $\|T_m(x,\cdot)-T(x,\cdot) \|_{L^2(\Omega)}=0$ for $|x|>\kappa/2$.
Therefore:
\begin{align}
\|H_m-H\|_{HS}^2 & = \int_\Omega \int_\Omega (T_m(x,y) - T(x,y))^2\, dx\,dy \\
& = \int_\Omega  \| T_m(x,\cdot) - T(x,\cdot) \|_{L^2(\Omega)}^2 \,dx \\
& \leq  \kappa C^2m^{-2s}\,dx \\
\end{align}
\end{proof}

As will be seen later, the convergence rate \eqref{eq:Fourier_rate} is optimal in the sense that no convolution-product expansion of order $m$ can achieve a better rate under the sole assumptions \ref{assumption1} and \ref{assumption2}.
\begin{corollary}
Let $\epsilon>0$  and set $m = \lceil C \epsilon^{-1/s} \kappa^{1/2s}\rceil$. Under assumptions \ref{assumption1} and \ref{assumption2}, $H_m$ satisfies $\| H - H_m \|_{HS} \leq \epsilon$ and products with $H_m$  and $H_m^*$ can be evaluated with no more than $O(\kappa^{1/2s} n \log n\epsilon^{-1/s} )$ operations.
\end{corollary}
\begin{proof}
Since Fourier atoms are not localized in the time domain, the modulation functions $\b{w}_k$ are supported on intervals of size $p = n$. 
The complexity of computing a matrix vector product is therefore $O( m n \log(n) )$ operations by lemma \ref{lem:complexity}.
\end{proof}

Finally, let us mention that computing the discrete Kohn-Nirenberg $\b{N}$ costs $O(\kappa n^2\log_2(n))$ operations ($\kappa n$ discrete Fourier transforms of size $n$). The storage cost of this Fourier representation is $O(m\kappa n)$ since one has to store $\kappa n$ coefficients for each of the $m$ vectors $\b{h}_k$.

In the next two sections, we show that replacing Fourier atoms by wavelet atoms or B-splines preserves the optimal rate of convergence in $O(\sqrt{\kappa} m^{-s})$, but has the additional advantage of being localized in space, thereby reducing complexity.

\begin{figure}
    \centering
    \begin{subfigure}[b]{0.3\textwidth}
        \includegraphics[width=\textwidth]{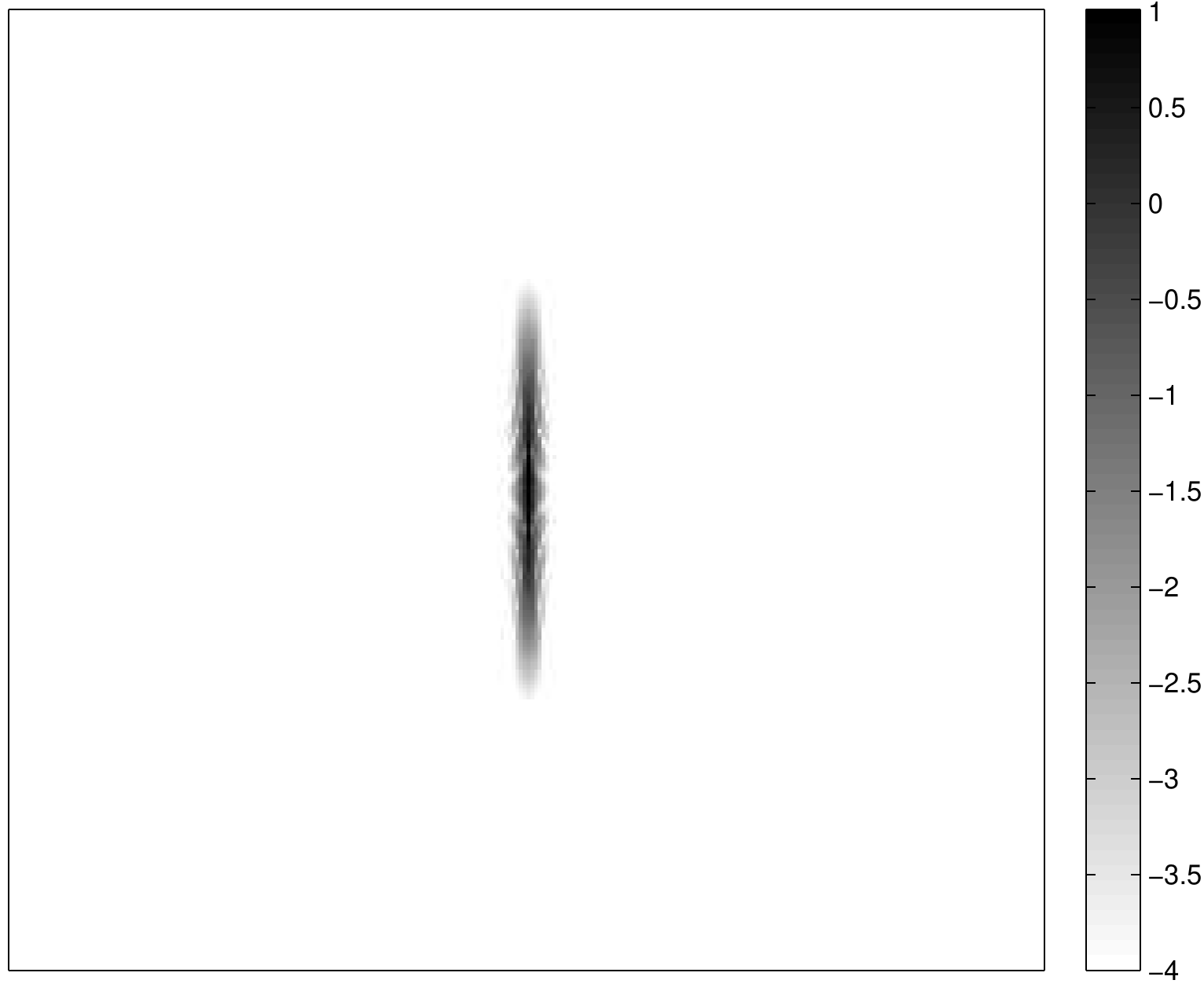}
        \caption{Kernel 1}
        \label{fig:k14}
    \end{subfigure}
    ~ 
    \begin{subfigure}[b]{0.3\textwidth}
        \includegraphics[width=\textwidth]{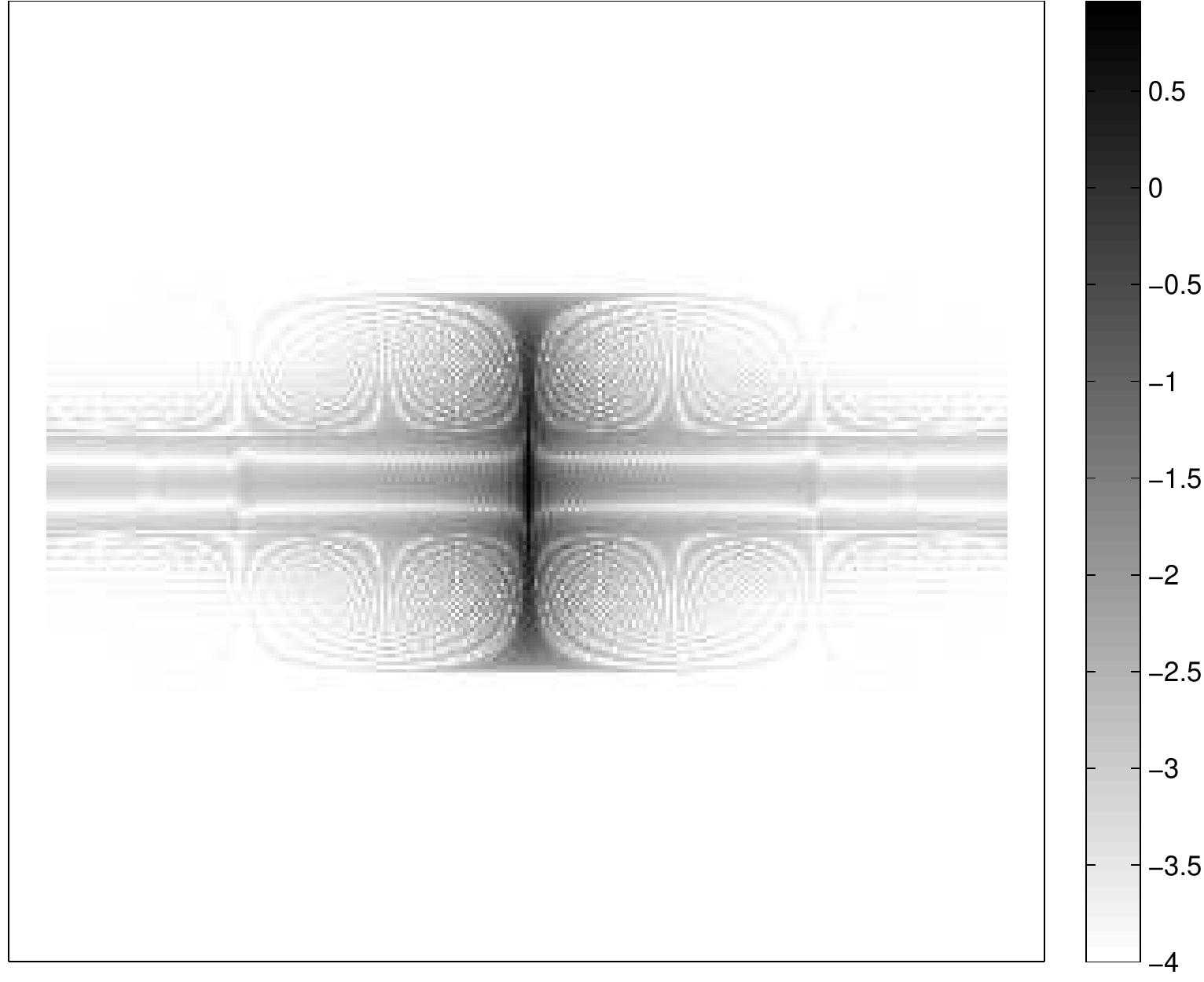}
        \caption{Kernel 2}
        \label{fig:k24}
    \end{subfigure}
    ~ 
    \begin{subfigure}[b]{0.3\textwidth}
        \includegraphics[width=\textwidth]{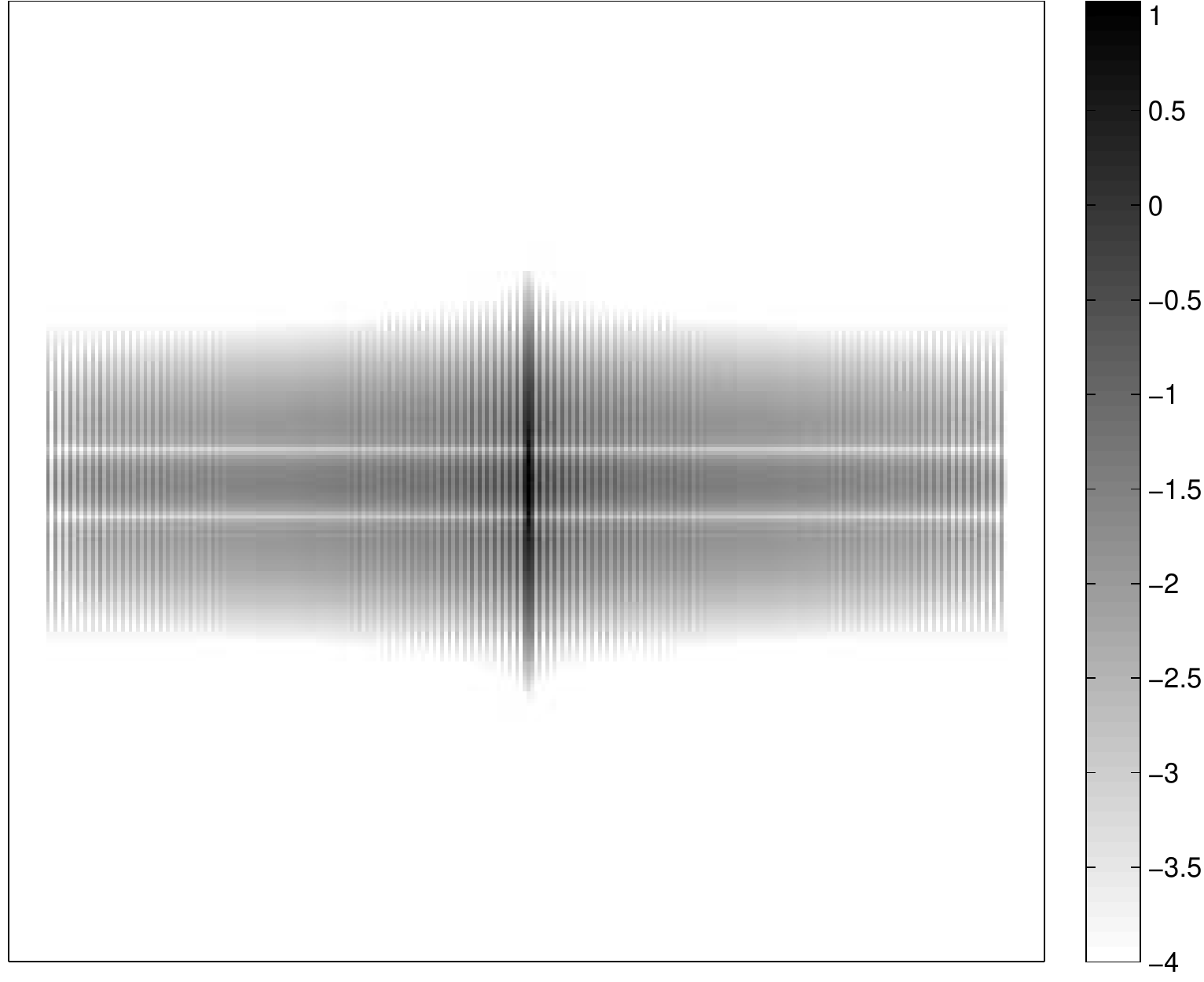}
        \caption{Kernel 3}
        \label{fig:k34}
    \end{subfigure}
    \caption{Kohn-Nirenberg symbols of the kernels given in examples \ref{example1}, \ref{example2} and \ref{example3} in $\log_{10}$ scale. Observe how the decay speed from the center (low frequencies) to the outer parts (high frequencies) changes depending on the TVIR smoothness. Note: the lowest values of the Kohn-Nirenberg symbol have been set to $10^{-4}$ for visualization purposes. } \label{fig:Kohn_Nirenberg} 
\end{figure}

\subsection{Spline decompositions}\label{sec:Spline}

\begin{theorem}[{\cite[p. 87]{bezhaev2001variational} or \cite[p. 420]{devore1993constructive}}]\label{lem:splineapprox}
 Let $f \in H^s(\Omega)$ and define its projection on $\mathcal{B}_{\alpha,m}$ by:
 \begin{align}
  f_m & = \argmin_{\tilde f \in \mathcal{B}_{\alpha,m}} \| \tilde f -f \|_2^2.
 \end{align}
If $\alpha\geq s$, then 
  \begin{equation}
    \| f - f_m \|_2 \leq C \sqrt{\kappa} m^{-s} \| f \|_{W^{s,2}}.
 \end{equation}
\end{theorem}
The following result directly follows.

\begin{corollary} \label{cor:spline_rate}
Set $\alpha \geq s$.
For each $x\in \Omega$, let $(c_k(x))_{0\leq k\leq m-1}$ be defined as the coefficients of the projection of $T(x,\cdot)$ on $\mathcal{B}_{\alpha,m}$:
\begin{equation}
 (c_k(x)) = \argmin_{(c_k)} \left\|T(x,\cdot) - \sum_{k=0}^{m-1} c_k B_{\alpha,m}(\cdot - k/m) \right\|_2^2.
\end{equation}
Define $T_m$ by:
\begin{equation} \label{eq:splproposedine_form}
 T_m(x,y)=\sum_{k=0}^{m-1} c_k(x) B_{\alpha,m}(y - k/m).
\end{equation}
If $\alpha \geq s$, then, under assumptions \ref{assumption1} and \ref{assumption2},
\begin{equation}\label{eq:spline_rate}
\|H_m - H\|_{HS}  \leq C \sqrt{\kappa} m^{-s}.
\end{equation}
\end{corollary}
\begin{proof}
 The proof is similar to that of corollary \eqref{thm:approx_Fourier}.
\end{proof}

\begin{corollary}\label{cor:complexitysplines}
Let $\epsilon>0$  and set $m = \lceil C \epsilon^{-1/s} \kappa^{1/2s} \rceil$. 
Under assumptions \ref{assumption1} and \ref{assumption2} $H_m$ satisfies $\| H - H_m \|_{HS} \leq \epsilon$ and products with $H_m$  and $H_m^*$ can be evaluated with no more than 
  \begin{equation}
   O\left( \left( s + \kappa^{1+1/2s} \epsilon^{-1/s} \right) n \log_2(\kappa n) \right)
  \end{equation}
  operations. For small $\epsilon$ and large $n$, the complexity behaves like
  \begin{equation}
   O\left( \kappa^{1+1/2s} n  \log_2 (\kappa n) \epsilon^{-1/s} \right).
  \end{equation}
\end{corollary}
\begin{proof}
In this approximation, $m$ B-splines are used to cover $\Omega$. B-splines have a compact support of size $(\alpha+1)/m$. 
This property leads to windowing vector $\b{w}_k$ with support of size $p = \lceil(\alpha+1)\frac{n}{m}\rceil$. 
Furthermore the vectors $(\b{h}_k)$ have a support of size $q = \kappa n$. Combining these two results with lemma \ref{lem:complexity} and corollary \ref{cor:spline_rate} yields the result for the choice $\alpha = s$.
\end{proof}

The complexity of computing the vectors $\b{c}_k$ is $O(\kappa n^2\log(n))$ ($\kappa n$ projections with complexity $n\log(n)$, see e.g. \cite{unser1993b}).

As can be seen in corollary \eqref{cor:complexitysplines}, B-splines approximations are preferable over Fourier decompositions whenever the support size $\kappa$ is small. 

\subsection{Wavelet decompositions}\label{sec:Wavelets}
 
 \begin{lemma}[{\cite[Theorem 9.5]{mallat1999wavelet}}]\label{lem:approx_rate_Wavelets}
 Let $f\in H^s(\Omega)$ and $f_m$ denote its partial wavelet series:
 \begin{equation}
  f_m=\sum_{|\mu| \leq \lceil \log_2(m) \rceil } c_\mu \psi_{\mu},
 \end{equation}
 where $\psi$ is a Daubechies wavelet with $\alpha > s$ vanishing moments and $c_\mu=\langle \psi_{\mu},f\rangle$. 
 Then 
 \begin{equation}
  \|f_m-f\|_{L^2(\Omega)} \leq C m^{-s} |f|_{H^s(\Omega)}.
 \end{equation}
\end{lemma}
A direct consequence is the following corollary.
\begin{corollary}\label{cor:approx_wavelets}
 Let $\psi$ be a Daubechies wavelet with $\alpha = s+1$ vanishing moments. Define $T_m$ by:
\begin{equation} \label{eq:wavelet_form}
 T_m(x,y)=\sum_{|\mu|\leq \lceil \log_2(m) \rceil} c_\mu(x) \psi_{\mu}(y),
\end{equation}
where $c_\mu(x) = \langle \psi_{\mu},T(x,\cdot) \rangle$. Then, under assumptions \ref{assumption1} and \ref{assumption2}
\begin{equation}\label{eq:Wavelet_rate}
\|H_m - H\|_{HS}  \leq C \sqrt{\kappa} m^{-s}.
\end{equation}
\end{corollary}
\begin{proof}
 The proof is identical to that of corollary \eqref{thm:approx_Fourier}.
\end{proof}

\begin{proposition}
Let $\epsilon>0$  and set $m = \lceil C \epsilon^{-1/s} \kappa^{1/2s} \rceil$. 
Under assumptions \ref{assumption1} and \ref{assumption2}
$H_m$ satisfies $\| H - H_m \|_{HS} \leq \epsilon$ and products with $H_m$ and $H_m^*$ can be evaluated with no more than
\begin{equation}
   O\left( \left( sn \log_2 \left( \epsilon^{-1/s} \kappa^{1/2s} \right) + \kappa^{1 + 1/2s} n \epsilon^{-1/s} \right) \log_2(\kappa n) \right)
  \end{equation}
operations. 
For small $\epsilon$, the complexity behaves like
\begin{equation}
   O\left( \kappa^{1+1/2s} n \log_2(\kappa n)  \epsilon^{-1/s} \right).
\end{equation}
\end{proposition}
\begin{proof}
 In \eqref{eq:wavelet_form}, the windowing vectors $\b{w}_k$ are wavelets $\b{\psi}_\mu$ of support of size $\min( (2s+1) n 2^{-|\mu|}, n )$. Therefore each convolution has to be performed on intervals of size $|\b{\psi}_\mu| + q + 1$. Since there are $2^{j}$ wavelets at scale $j$, the total number of operations is:
 \begin{align}
 & \sum_{\mu \, | \, |\mu| < \log_2(m)} ( |\b{\psi}_\mu| + q + 1 ) \log_2( \min( |\b{\psi}_\mu|, q+1) ) \\
 \leq &\sum_{\mu \, | \, |\mu| < \log_2(m)} ( (2s+1) n 2^{-|\mu|} + \kappa n ) \log_2( \kappa n ) \\
  = &\sum_{j=0}^{\log_2(m) - 1} 2^j \left( (2s+1)n 2^{-j} + \kappa n\right) \log_2( \kappa n)\\
  = &\sum_{j=0}^{\log_2(m) - 1} \left( (2s+1)n + 2^j \kappa n \right) \log_2( \kappa n) \\
  \leq &\left( (2s+1)n \log_2(m) + m \kappa n\right) \log_2( \kappa n ) \\
  = & \left( (2s+1)n \log_2(\epsilon^{-1/s} \kappa^{1/2s}) + \epsilon^{-1/s} \kappa^{1 + 1/2s} n\right) \log_2( \kappa n ).
 \end{align}
\end{proof}

\begin{figure}
    \centering
    \begin{subfigure}[b]{0.3\textwidth}
        \includegraphics[width=\textwidth]{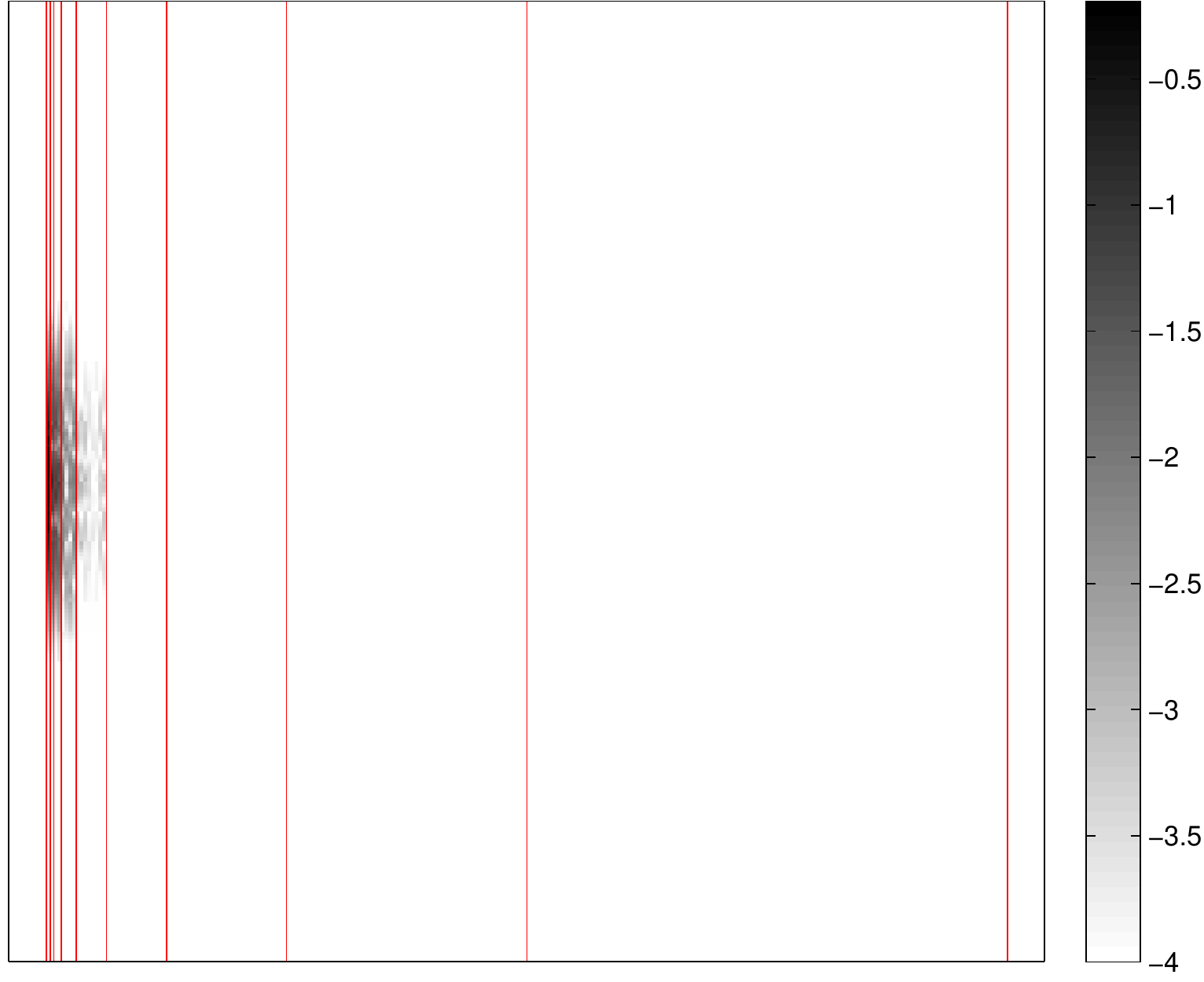}
        \caption{Kernel 1}
        \label{fig:k15}
    \end{subfigure}
    ~ 
    \begin{subfigure}[b]{0.3\textwidth}
        \includegraphics[width=\textwidth]{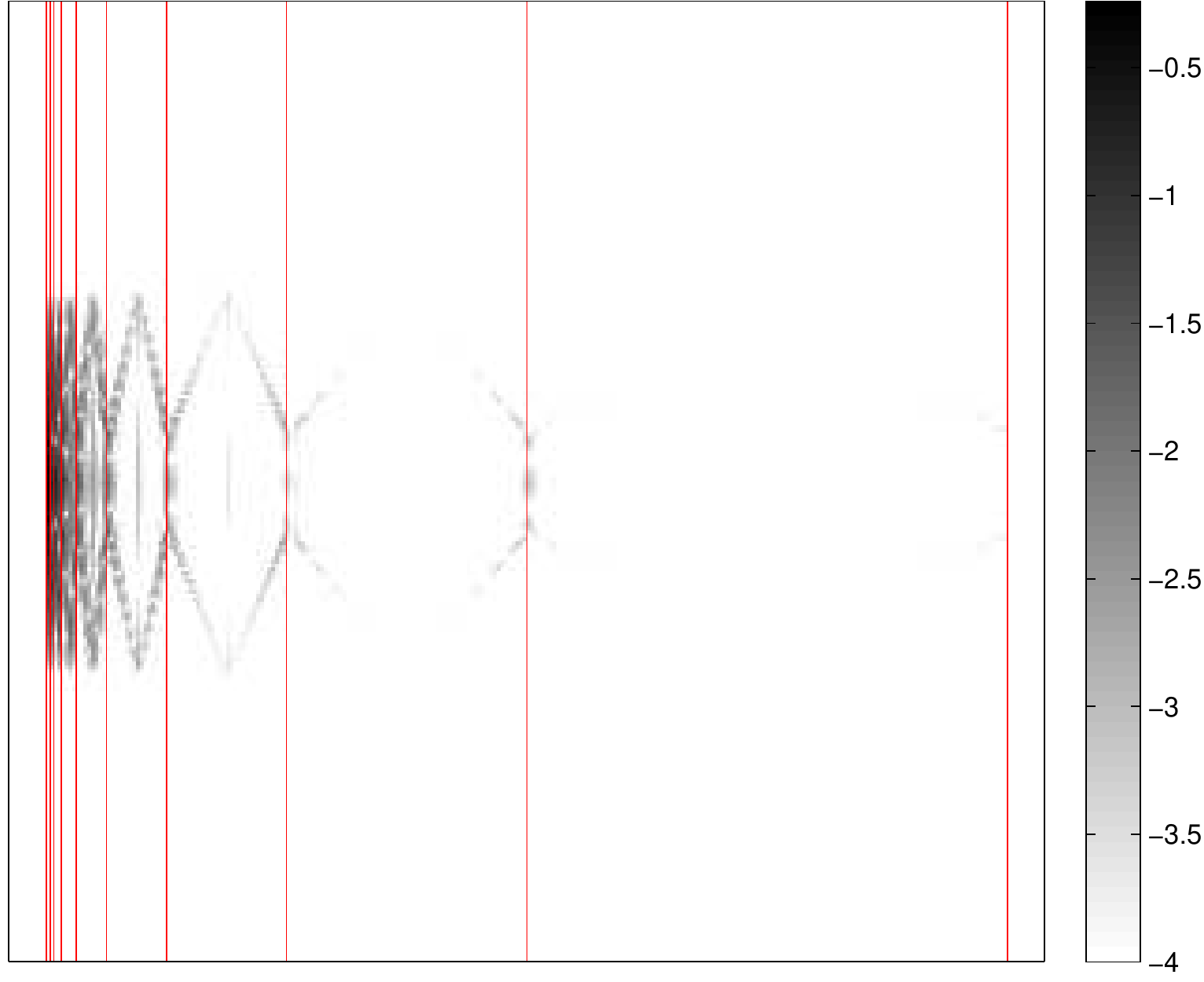}
        \caption{Kernel 2}
        \label{fig:k25}
    \end{subfigure}
    ~ 
    \begin{subfigure}[b]{0.3\textwidth}
        \includegraphics[width=\textwidth]{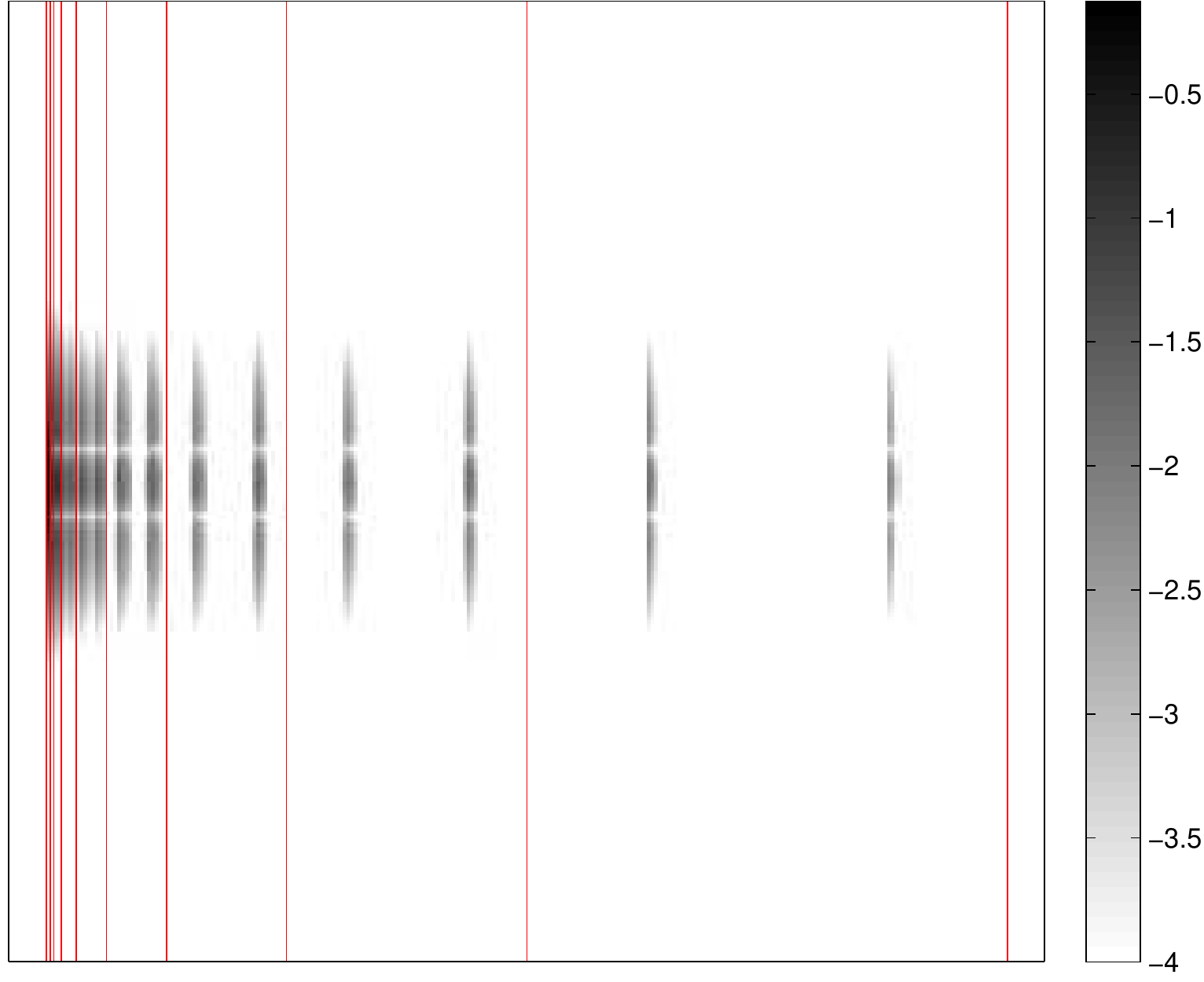}
        \caption{Kernel 3}
        \label{fig:k35}
    \end{subfigure}
    \caption{``Wavelet symbols'' of the operators given in examples \ref{example1}, \ref{example2} and \ref{example3} in $\log_{10}$ scale. The red bars indicate separations between scales. Notice that the wavelet coefficients in kernel 1 rapidly decay as scales increase. The decay is slower for kernels 2 and 3 which are less regular. The adaptivity of wavelets can be visualized in Kernel 3: some wavelet coefficients are non zero at large scales, but they are all concentrated around discontinuities. Therefore only a few number of couples $(c_\mu,\psi_\mu)$ will be necessary to encode the discontinuities. This was not the case with Fourier or B-spline atoms.} \label{fig:Wavelet_Symbol} 
\end{figure}

Computing the vectors $\b{c}_\mu$ costs $O(\kappa s n^2)$ operations ($\kappa n$ discrete wavelet transforms of size $n$). 
The storage cost of this wavelet representation is $O(m\kappa n)$ since one has to store $\kappa n$ coefficients for each of the $m$ functions $\b{h}_k$.

As can be seen from this analysis, wavelet and B-spline approximations roughly have the same complexity over the class $\mathcal{T}^s$. 
The first advantage of wavelets compared to B-splines is that the coefficients $c_\mu(x)$ have a simple analytic expression, while B-splines coefficients $c_k$ are found by solving a linear system. This is slightly more complicated to implement.

The second significant advantage of wavelets compared to B-splines with fixed knots is that they are known to characterize much more general function spaces than $H^s(\Omega)$. 
For instance, if all functions $T(x,\cdot)$ have a single discontinuity at a given $y\in \Omega$, only a few coefficients $c_\mu(x)$ will remain of large amplitude. Wavelets will be able to efficiently encode the discontinuity, while B-splines with fixed knots - which are not localized in nature - will fail to well approximate the TVIR. 
It is therefore possible to use wavelets in an adaptive way. This effect is visible on figure \ref{fig:k35}: despite discontinuities, only wavelets localized around the discontinuities yield large coefficients.
In the next section, we propose two other adaptive methods, in the sense that they are able to automatically adapt to the TVIR regularity.

\subsection{Interpolation VS approximation}

In all previous results, we constructed the functions $w_k$ and $h_k$ in \ref{eq:convolutionproduct} by projecting $T(x,\cdot)$ onto linear subspaces. This is only possible if the whole TVIR $T$ is available. 
In very large scale applications, this assumption is unrealistic, since the TVIR contains $n^2$ coefficients, which cannot even be stored. 
Instead of assuming a full knowledge of $T$, some authors (e.g. \cite{nagy1998restoring}) assume that the impulse responses $T(\cdot, y)$ are available only at a discrete set of points $y_i=i/m$ for $1\leq i\leq m$.

In that case, it is possible to \emph{interpolate} the impulse responses instead of approximating them. Given a linear subspace $E_m=\mathrm{span}(e_k, k\in \{1,\hdots, m\})$, where the atoms $e_k$ are assumed to be linearly independent, the functions $c_k(x)$ in \eqref{eq:decomposition_linear_subspace} are chosen by solving the set of linear systems:
\begin{equation}
 \sum_{k=1}^m c_k(x) e_k(y_i)  = T_m(x,y_i)  \quad \textrm{for} \quad 1\leq i \leq m.
\end{equation}
In the discrete setting, under assumption \ref{assumption2}, this amounts to solving $\lceil \kappa n \rceil$ linear systems of size $m\times m$. 
We do not discuss the rates of approximation for this interpolation technique since they are usually expressed in the $L^\infty$-norm under more stringent smoothness assumptions than $T\in \mathcal{T}^s$. 
We refer the interested reader to \cite{schoenberg1973cardinal,devore1993constructive,donoho1992interpolating} for results on spline and wavelet interpolants.

\subsection{On Meyer's operator representation}\label{sec:Beylkin}

Up to now, we only assumed a regularity of $T$ in the $y$ direction, meaning that the impulse responses vary smoothly in space. 
In many applications, the impulse responses themselves are smooth. 
In this section, we show that this additional regularity assumption can be used to further compress the operator. 
Finding a compact operator representation is a key to treat identification problems (e.g. blind deblurring in imaging).

Since $(\psi_{\lambda})_{\lambda\in \Lambda}$ is a Hilbert basis of $L^2(\Omega)$, the set of tensor product functions $(\psi_\lambda \otimes \psi_\mu)_{\lambda\in \Lambda,\mu \in \Lambda}$ is a Hilbert basis of $L^2(\Omega\times \Omega)$.
Therefore, any $T\in L^2(\Omega\times \Omega)$ can be expanded as:
\begin{equation}
 T(x,y) = \sum_{\lambda\in \Lambda} \sum_{\mu \in \Lambda} c_{\lambda,\mu} \psi_\lambda(x) \psi_\mu(y).
\end{equation}
The main idea of the construction in this section consists of keeping only the coefficients $c_{\lambda,\mu}$ of large amplitude.
A similar idea was proposed in the BCR paper \cite{beylkin1991fast}\footnote{This was also the basic idea in our recent paper \cite{escande2015sparse}.}, except that the kernel $K$ was expanded instead of the TVIR $T$. 
Decomposing $T$ was suggested by Beylkin at the end of \cite{beylkin1992representation} without a precise analysis.

In this section, we assume that $T\in H^{r,s}(\Omega\times \Omega)$, where 
\begin{equation}
 H^{r,s}(\Omega\times \Omega)=\{T:\Omega\times \Omega\to \R,\ \partial_{x}^{\alpha_1}\partial_{y}^{\alpha_2}T \in L^2(\Omega\times \Omega),\ \forall \alpha_1\in \{0,\hdots, r\}, \forall \alpha_2\in \{0,\hdots, s\} \}.
\end{equation}
This space arises naturally in applications, where the impulse response regularity $r$ might differ from the regularity $s$ of their variations. Notice that $H^{2s}(\Omega\times \Omega)\subset H^{s,s}(\Omega\times \Omega) \subset H^{s}(\Omega)$. 
\begin{theorem}\label{thm:Meyer}
 Assume that $T\in H^{r,s}(\Omega\times \Omega)$ and satisfies assumption \ref{assumption2}. Assume that $\psi$ has $\max(r,s)+1$ vanishing moments.
 Let $c_{\lambda,\mu}= \langle T, \psi_\lambda\otimes \psi_\mu\rangle$. Define
 \begin{equation}
   H_{m_1,m_2}  = \sum_{|\lambda|\leq \log_2(m_1)} \sum_{|\mu|\leq \log_2(m_2)} c_{\lambda,\mu} \psi_{\lambda}\otimes \psi_\mu.
 \end{equation}
 Let $m\in \N$, set $m_1 =\lceil m^{s/(r+s)}\rceil$, $m_2=\lceil m^{r/(r+s)}\rceil$ and $H_m=H_{m_1,m_2}$. Then
 \begin{equation}
  \|H-H_m\|_{HS} \leq C \sqrt{\kappa} m^{-\frac{rs}{r+s}}.\label{eq:boundmeyer}
 \end{equation}
\end{theorem}
\begin{proof}
First notice that 
\begin{equation}
 T_{\infty, m_2} = \sum_{|\mu| \leq \lceil \log_2(m_2) \rceil} c_\mu \otimes \psi_\mu,
\end{equation}
where $c_\mu(x) = \langle T(x,\cdot) , \psi_\mu\rangle$.
From corollary \ref{cor:approx_wavelets}, we get:
\begin{equation}
 \|T_{\infty, m_2} - T\|_{L^2(\Omega\times \Omega)} \leq  C \sqrt{\kappa} m_2^{-s}.
\end{equation}
Now, notice that $c_\mu\in H^r(\Omega)$.
Indeed, for all $0 \leq k \leq r$, we get:
 \begin{align}
 & \int_{\Omega} (\partial_x^k c_{\mu}(x)) ^2 \,dx \\
 &= \int_{\Omega} \left(\partial_x^k \int_\Omega T(x,y) \psi_\mu(y)\, dy \right) ^2 \,dx \\
 &= \int_{\Omega} \left(\int_\Omega (\partial_x^k T)(x,y) \psi_\mu(y)\, dy \right) ^2 \,dx \\ 
 &\leq \int_{\Omega}  \| (\partial_x^k T)(x,\cdot)\|_{L^2(\Omega)}^2 \|\psi_\mu\|_{L^2(\Omega)}^2 \,dx \\ 
 &= \| (\partial_x^k T) \|_{L^2(\Omega\times \Omega)}<+\infty.
\end{align}
Therefore, we can use lemma \ref{lem:approx_rate_Wavelets} again to show:
\begin{equation}
 \|T_{\infty,m_2} - T_{m_1,m_2}\|_{L^2(\Omega\times \Omega)} \leq  C \sqrt{\kappa} m_1^{-r}.
\end{equation}
Finally, using the triangle inequality, we get:
\begin{equation}\label{eq:bbbbbb}
 \|T - T_{m_1,m_2}\|_{HS} \leq C\sqrt{\kappa}(m_1^{-r} + m_2^{-s}).
\end{equation}
By setting $m_1=m_2^{s/r}$, the two approximation errors in the right-hand side of \eqref{eq:bbbbbb} are balanced. This motivates the choice of $m_1$ and $m_2$ indicated in the theorem.
\end{proof}
The approximation result in inequality \eqref{eq:boundmeyer} is worst than the previous ones. 
For instance if $r=s$, then the bound becomes $O(\sqrt{\kappa} m^{-s/2})$ instead of $O(\sqrt{\kappa}m^{-s})$ in all previous theorems. 
The great advantage of this representation is the operator storage: until now, the whole set of vectors $(\b{c}_\mu)$ had to be stored ($O(\kappa n m)$ values), while now, only $m$ coefficients $c_{\lambda,\mu}$ are required. 
For instance, in the case $r=s$, for an equivalent precision, the storage cost of the new representation is $O(\kappa m^2)$ instead of $O(\kappa nm)$.

In addition, evaluating matrix-vector products can be achieved rapidly by using the following trick:
\begin{align}
 \b{H_m} \b{u} &= \sum_{|\lambda|\leq \log_2(m_1)} \sum_{|\mu|\leq \log_2(m_2)} c_{\lambda,\mu} \b{\psi}_{\lambda} \star (\b{\psi}_{\mu} \odot \b{u}) \\
 &= \sum_{|\mu|\leq \log_2(m_2)} \left(\sum_{|\lambda|\leq \log_2(m_1)}  c_{\lambda,\mu} \b{\psi}_{\lambda}\right) \star (\b{\psi}_{\mu} \odot \b{u}).
\end{align}
By letting $\tilde{\b{c}}_\mu = \sum_{|\lambda|\leq \log_2(m_1)}  c_{\lambda,\mu} \b{\psi}_{\lambda}$, we get
\begin{equation}\label{eq:trickMeyer}
 \b{H}_m \b{u} = \sum_{|\mu|\leq \log_2(m_2)} \tilde{\b{c}}_\mu \star (\b{\psi}_\mu \odot \b{u}).
\end{equation}
which can be can be computed in $O(m_2 \kappa n \log_2(\kappa n))$ operations. This remark leads to the following proposition.
\begin{proposition}
 Assume that $T\in H^{r,s}(\Omega \times \Omega)$ and that it satisfies assumption \ref{assumption2}. 
 Set $m=\left\lceil \left( \frac{\epsilon}{C\sqrt{\kappa}} \right)^{-(r+s)/rs}\right\rceil$. Then the operator $H_m$ defined in theorem \ref{thm:Meyer} satisfies $\|H-H_m\|_{HS}\leq \epsilon$ and the number of operations necessary to evaluate a product with $H_m$ or $H_m^*$ is bounded above by $O\left(\epsilon^{-1/s} \kappa^{\frac{2s+1}{2s}} n \log_2(n)\right)$.
\end{proposition}

Notice that the complexity of matrix-vector products is unchanged compared to the wavelet or spline approaches with a much better compression ability. However, this method requires a preprocessing to compute $\tilde{\b{c}}_\mu$ with complexity $\epsilon^{-1/s} \kappa^{1/2s} n$.

\begin{figure}
    \centering
    \begin{subfigure}[b]{0.3\textwidth}
        \includegraphics[width=\textwidth]{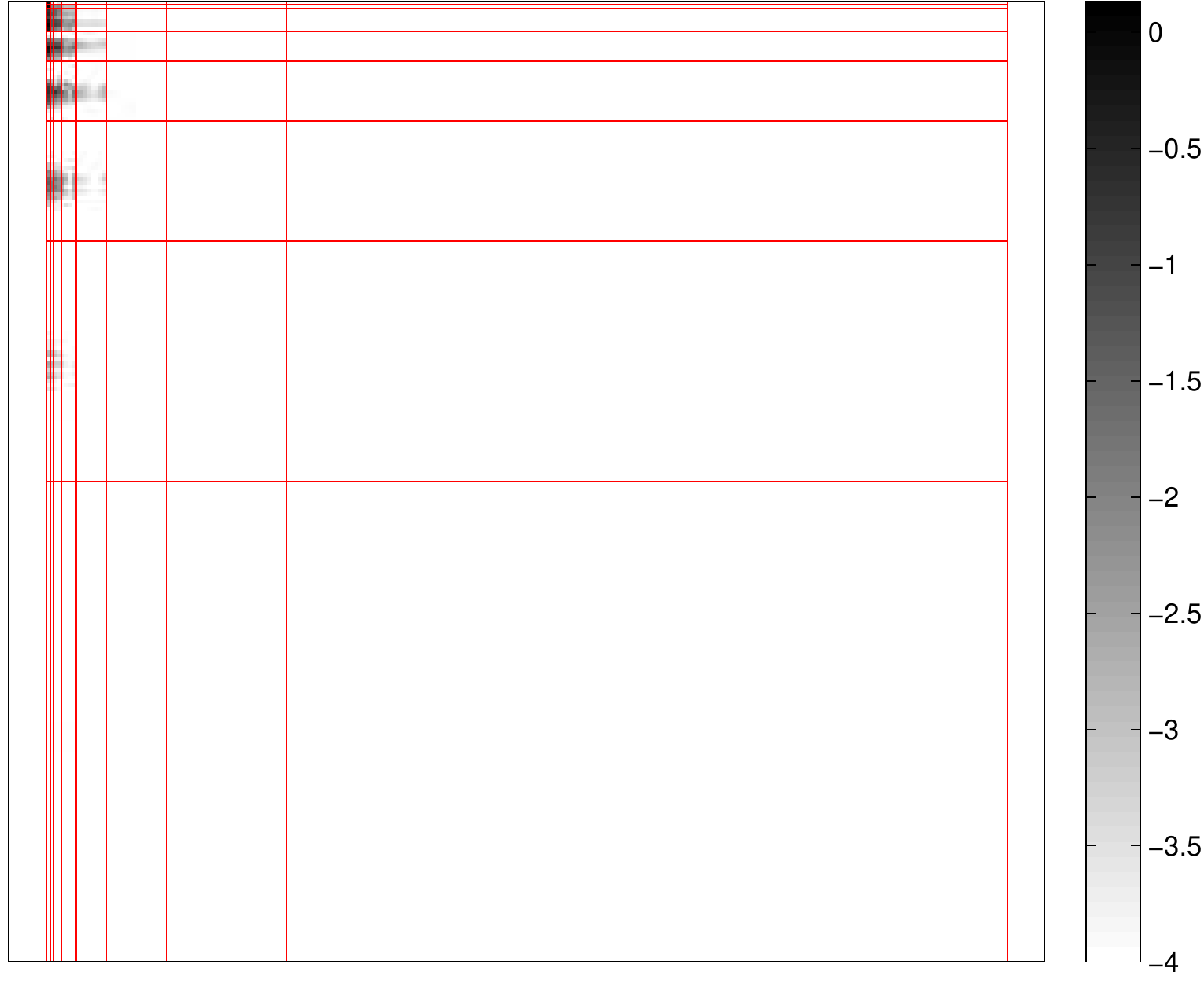}
        \caption{Kernel 1}
        \label{fig:k16}
    \end{subfigure}
    ~ 
    \begin{subfigure}[b]{0.3\textwidth}
        \includegraphics[width=\textwidth]{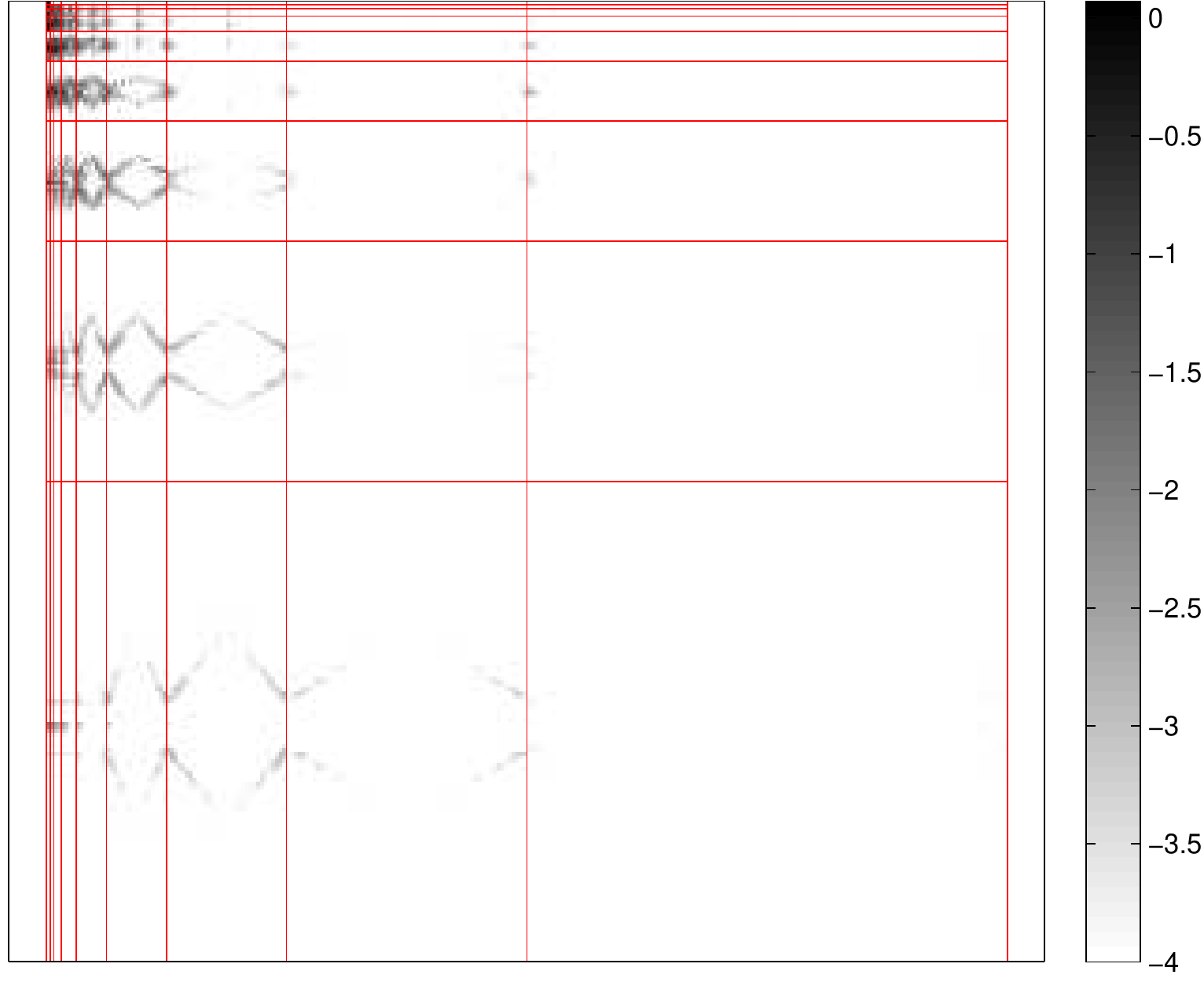}
        \caption{Kernel 2}
        \label{fig:k26}
    \end{subfigure}
    ~ 
    \begin{subfigure}[b]{0.3\textwidth}
        \includegraphics[width=\textwidth]{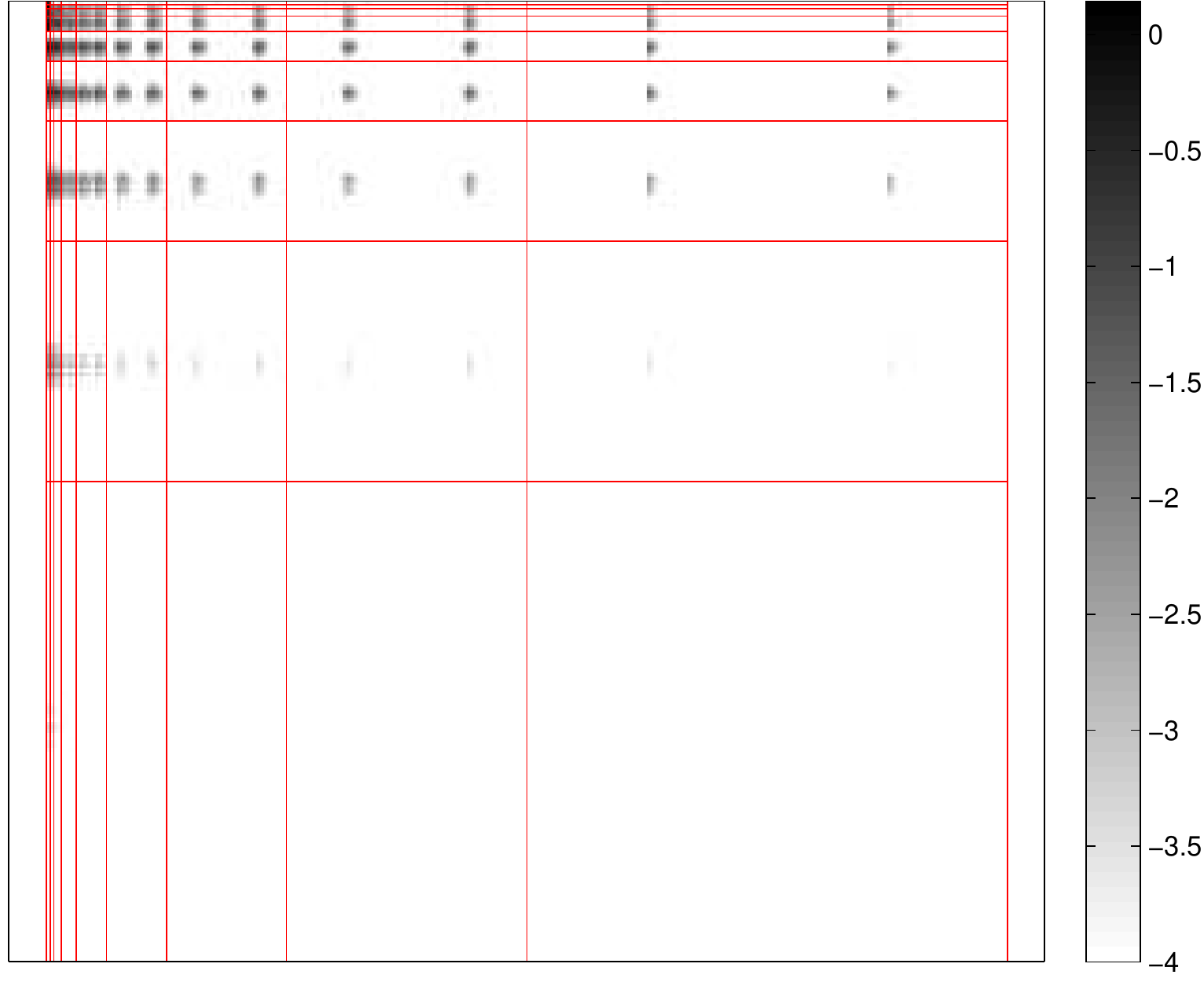}
        \caption{Kernel 3}
        \label{fig:k36}
    \end{subfigure}
    \caption{Meyer's representations of the operators in examples \ref{example1}, \ref{example2} and \ref{example3} in $\log_{10}$ scale.} \label{fig:Beylkin_Symbol} 
\end{figure}
\section{Adaptive decompositions}\label{sec:adaptive}

In the last section, all methods shared the same principle: project $T(x,\cdot)$ on a fixed basis for each $x\in \Omega$. 
Instead of fixing a basis, one can try to find a basis adapted to the operator at hand.
This idea was proposed in \cite{flicker2005anisoplanatic} and \cite{denis2015fast}. 

\subsection{Singular value decompositions}\label{sec:SVD}

The authors of \cite{flicker2005anisoplanatic} proposed to use a singular value decomposition (SVD) of the TVIR in order to construct the functions $h_k$ and $w_k$. 
In this section we first detail this idea and then analyze it from an approximation theoretic point of view.
Let $J:L^2(\Omega)\to L^2(\Omega)$ denote the linear integral operator with kernel $T\in \mathcal{T}^s$. 
First notice that $J$ is a Hilbert-Schmidt operator since $\|J\|_{HS}=\|H\|_{HS}$.
By lemma \ref{lem:schmidt} and since Hilbert-Schmidt operators are compact, there exists two Hilbert bases $(e_k)$ and $(f_k)$ of $L^2(\Omega)$ such that $J$ can be decomposed as
\begin{equation}
  J = \sum_{k\geq 1} \sigma_k \cdot e_k \otimes f_k,
\end{equation}
leading to
\begin{equation}\label{eq:decomposition_svd_T} 
 T(x,y)= \sum_{k=1}^{+\infty} \sigma_k f_k(x) e_k(y).
\end{equation}

The following result is a standard.
\begin{theorem}
For a given $m$, a set of functions $(h_k)_{1\leq k\leq m }$ and $(w_k)_{1\leq k\leq m }$ that minimizes $\|H_m - H\|_{HS}$ is given by:
\begin{equation} 
h_k = \sigma_k f_k \quad \textrm{and} \quad w_k = e_k.
\end{equation}
Moreover, if $T(x,\cdot)$ satisfies assumptions \ref{assumption1} and \ref{assumption2}, we get:
\begin{equation} \label{eq:svd_rate}
\|H_m - H\|_{HS} = O\left( \sqrt{\kappa} m^{-s} \right).
\end{equation}
\end{theorem}
\begin{proof}
 The proof of optimality \eqref{eq:svd_rate} is standard. 
%
%
Since $T_m$ is the best rank $m$ approximation of $T$, it is necessarily better than bound \eqref{eq:Fourier_rate}, yielding \eqref{eq:svd_rate}.
\end{proof}

\begin{theorem}
For all $\epsilon > 0$ and $m<n$, there exists an operator $H$ with TVIR satisfying \ref{assumption1} and \ref{assumption2} such that:
 \begin{equation}\label{eq:svd_lowerbound}
  \|H_m - H\|_{HS} \geq C \sqrt{\kappa} m^{-(s+\epsilon)}.
 \end{equation}\end{theorem}
\begin{proof}
In order to prove, \eqref{eq:svd_lowerbound}, we construct a ``worst case'' TVIR $T$. 
We first begin by constructing a kernel $T$ with $\kappa=1$ to show a simple pathological TVIR. 
Define $T$ by:
 \begin{equation}\label{eq:svd_construction}
  T(x,y) = \sum_{k \in \Z} \sigma_k f_k(x) f_k(y),
 \end{equation}
 where $f_k(x) = \exp(2i\pi k x)$ is the $k$-th Fourier atom, $\sigma_0=0$ and $\sigma_k = \sigma_{-k} = \frac{1}{|k|^{s+1/2+\epsilon/2}}$ for $|k|\geq 1$. With this choice, 
  \begin{equation}
  T(x,y) = \sum_{|k| \leq N} 2 \sigma_k \cos(2\pi (x+y))
 \end{equation}
 is real for all $(x,y)$.
 We now prove that $T\in \mathcal{T}^s$. 
 The $k$-th Fourier coefficient of $T(x,\cdot)$ is given by $\sigma_k f_k(x)$ which is bounded by $\sigma_k$ for all $x$. 
 By lemma \ref{lem:Characterization_Fourier_Sobolev}, $T(x,\cdot)$ therefore belongs to $H^{s}(\Omega)$ for all $x\in \Omega$. 
 By construction, the spectrum of $T$ is $(|\sigma_k|)_{k\in \N}$, therefore for any rank $2m+1$ approximation of $T$, we get:
 \begin{align}\label{eq:list_of_estimates}
  \|T-T_{2m+1}\|_{HS}^2 & \geq \sum_{|k|\geq m+1} \frac{1}{|k|^{2s+1+\epsilon}} \\
  & \geq  \int_{m+1}^\infty \frac{2}{t^{2s+1+\epsilon}} \,dt \\
  & = \frac{1}{2s+\epsilon} \frac{2}{(m+1)^{2s+\epsilon}} \\
  &= O(m^{-2s-\epsilon}),
  \end{align}
	proving the result for $\kappa=1$. Notice that the kernel $K$ of the operator with TVIR $T$ only depends on $x$:
\begin{equation}
  K(x,y) = \sum_{|k| \leq N} 2 \sigma_k \cos(2\pi x).
\end{equation}
Therefore the worst case TVIR exhibited here is that of a \emph{rank 1} operator $H$. 
Obviously, it cannot be well approximated by product-convolution expansions. 

Let us now construct a TVIR satisfying assumption \ref{assumption2}. For this, we first construct an orthonormal basis $(\tilde f_k)_{k\in \Z}$ of $L^2([-\kappa/2,\kappa/2])$ defined by:
\begin{equation}
 \tilde f_k(x) = \left\{\begin{array}{ll}
			   \frac{1}{\sqrt{\kappa}} f_k\left(\frac{x}{\kappa}\right) & \textrm{if}\ |x|\leq \frac{\kappa}{2}, \\
			   0 & \textrm{otherwise}.                  
                 \end{array}\right.
\end{equation}
The worst case operator considered now is defined by:
\begin{equation}
 T(x,y) = \sum_{k\in \Z} \tilde \sigma_k \tilde f_k(x) f_k(y). 
\end{equation}
Its spectrum is $(|\tilde \sigma_k|)_{k\in \Z}$, and we get 
\begin{equation}
|\langle T(x,\cdot), f_k \rangle| = |\tilde \sigma_k \tilde f_k(x)| = \frac{1}{\kappa} |\tilde \sigma_k|.
\end{equation}
By lemma \ref{lem:approx_rate_Fourier}, if $\tilde \sigma_k = \frac{\kappa}{(1+|k|^2)^s|k|^{1+\epsilon}}$, then $\|T(x,\cdot)\|_{H^s(\Omega)}$ is uniformly bounded by a constant independent of $\kappa$. Moreover, by reproducing the reasoning in \eqref{eq:list_of_estimates}, we get:
\begin{equation}
 \|T-T_{2m+1}\|_{HS}^2 = O(\kappa m^{-2s-\epsilon}).
\end{equation}

%
%
\end{proof}

Even if the SVD provides an optimal decomposition, there is no guarantee that functions $e_k$ are supported on an interval of small size. 
As an example, it suffices to consider the ``worst case'' TVIR given in equation \eqref{eq:svd_construction}.
Therefore, vectors $\b{w}_k$ are generically supported on intervals of size $p = n$. This yields the following proposition.

\begin{corollary}
Let $\epsilon>0$  and set $m = \lceil C \epsilon^{-1/s} \kappa^{1/2s}\rceil$. Then $H_m$ satisfies $\| H - H_m \|_{HS} \leq \epsilon$ and a product with $H_m$ and $H_m^*$ can be evaluated with no more than $O( \kappa^{1/2s} n \log n\epsilon^{-1/s})$ operations.
\end{corollary}

Computing the first $m$ singular vectors in \eqref{eq:decomposition_svd_T} can be achieved in roughly $O(\kappa n^2\log(m))$ operations thanks to recent advances in randomized algorithms \cite{halko2011finding}. The storage cost for this approach is $O(mn)$ since the vectors $\b{e}_k$ have no reason to be compactly supported.

\subsection{The optimization approach in \cite{denis2015fast}}\label{sec:Denis}

In \cite{denis2015fast}, the authors propose to construct the windowing functions $w_k$ and the filters $h_k$ using constrained optimization procedures. For a fixed $m$, they propose solving:
\begin{equation}\label{eq:structuredlowrank}
 \min_{(h_k,w_k)_{1\leq k \leq m}} \left\|T - \sum_{k=1}^m h_k\otimes w_k \right\|_{HS}^2
\end{equation}
under an additional constraint that $\mathrm{supp}(w_k)\subset \omega_k$ with $\omega_k$ chosen so that $\cup_{k=1}^m \omega_k =\Omega$.
A decomposition of type \ref{eq:structuredlowrank} is known as structured low rank approximation \cite{chu2003structured}.
This problem is non convex and to the best of our knowledge, there currently exists no algorithm running in a reasonable time to find its global minimizer. It can however be solved approximately using alternating minimization like algorithms. 

Depending on the choice of the supports $\omega_k$, different convergence rates can be expected. However, by using the results for B-splines in section \ref{sec:Spline}, we obtain the following proposition.
\begin{proposition}
Set $\omega_k=[(k-1)/m,k/m+s/m]$ and let $(h_k,w_k)_{1\leq k \leq m}$ denote the global minimizer of \eqref{eq:structuredlowrank}. 
Define $T_m$ by $T_m(x,y) = \sum_{k=1}^m h_k(x) w_k(y)$. Then:
\begin{equation}
\|T - T_m\|_{HS}^2 \leq C \sqrt{\kappa}m^{-s}.
\end{equation}
Set $m=\lceil \kappa^{1/2s} C\epsilon^{-1/s}\rceil$, then $\|H_m-H\|_{HS}\leq \epsilon$ and the evaluation of a product with $H_m$ or $H_m^*$ is of order
\begin{equation}
 O(\kappa^{1 + 1/2s} n\log(n) \epsilon^{-1/s}).
\end{equation}
\end{proposition}
\begin{proof}
First notice that cardinal B-Splines are also supported on $[(k-1)/m,k/m+s/m]$.  
Since the method in \cite{denis2015fast} provides the best choices for $(h_k,w_k)$, the distance $\|H_m-H\|_{HS}$ is necessarily lower than that obtained using B-splines in corollary \ref{cor:spline_rate}.
\end{proof}

Finally, let us mention that - owing to corollary \ref{cor:approx_wavelets} - it might be interesting to use the optimization approach \eqref{eq:structuredlowrank} with windows of varying sizes. 

\section{Summary and extensions}\label{sec:additional}

\subsection{A summary of all results}

Table  \ref{table:table} summarizes the results derived so far under assumptions \ref{assumption1} and \ref{assumption2}. In the particular case of Meyer's methods, we assume that $T\in H^{r,s}(\Omega\times \Omega)$ instead of assumption \ref{assumption1}.
As can be seen in this table, different methods should be used depending on the application. 
The best methods are:
\begin{itemize}
 \item Wavelets: they are adaptive, have a relatively low construction complexity, and matrix-vector products also have the best complexity. 
 \item Meyer: this method has a big advantage in terms of storage. The operator can be represented very compactly with this approach. It has a good potential for problems where the operator should be inferred (e.g. blind deblurring). It however requires stronger regularity assumptions.
 \item The SVD and the method proposed in \cite{denis2015fast} both share an optimal adaptivity. The representation however depends on the operator and it is more costly to evaluate it.
\end{itemize}

\begin{center}
\begin{table}[h]
    \begin{tabu}{|[1pt]c|[1pt] c |[1pt] c |[1pt] c |[1pt] c |[1pt] c|[1pt]}
    \tabucline[1pt]{-}
    \emph{Method} & \emph{Approximation} & \emph{Product} & \emph{Construction} & \emph{Storage} & \emph{Adaptivity} \\ \tabucline[1pt]{-}
    Fourier \ref{sec:Fourier} & $O\left(\kappa^{\frac{1}{2}} m^{-s}\right)$ & $O\left(\kappa^{\frac{1}{2s}} n\log(n) \epsilon^{-\frac{1}{s}}\right)$ & $O(\kappa n^2\log(n))$ & $O(m\kappa n)$ & \xmark \\ \hline
    B-Splines \ref{sec:Spline} & $O\left(\kappa^{\frac{1}{2}}m^{-s}\right)$ & $O\left(\kappa^{\frac{2s+1}{2s}} n\log(n) \epsilon^{-\frac{1}{s}}\right)$ & $O(\kappa n^2\log(n))$ & $O(m\kappa n)$ & \xmark \\ \hline
    Wavelets \ref{sec:Wavelets} & $O\left(\kappa^{\frac{1}{2}}m^{-s}\right)$ & $O\left(\kappa^{\frac{2s+1}{2s}} n\log(n) \epsilon^{-\frac{1}{s}}\right)$ & $O(\kappa s n^2)$ & $O(m\kappa n)$ & \cmark \\ \hline
    Meyer \ref{sec:Beylkin} &  $O\left(\kappa^{\frac{1}{2}} m^{-\frac{rs}{r+s}}\right)$ & $O\left(\kappa^{\frac{2s+1}{2s}} n \log(n)\epsilon^{-\frac{1}{s}}\right)$ & $O(s n^2)$ & $ O(m)$ & \cmark \\ \hline
    SVD \ref{sec:SVD} & $O\left(\kappa^{\frac{1}{2}}m^{-s}\right)$ & $O\left(\kappa^{\frac{1}{2s}} n\log(n) \epsilon^{-\frac{1}{s}}\right)$ & $O(\kappa n^2\log(m))$ & $O(m n)$ & \cmark  \\ \hline
    \cite{denis2015fast} \ref{sec:Denis} & $O\left(\kappa^{\frac{1}{2}}m^{-s}\right)$ & $O\left(\kappa^{\frac{2s+1}{2s}} n\log(n) \epsilon^{-\frac{1}{s}}\right)$ & High (iterative) & $O(m \kappa n)$ & \cmark  \\ \tabucline[1pt]{-}
    \end{tabu}
    \caption{Summary of the properties of different constructions. Approximation $\equiv$ approximation rates in terms of $m$. Product $\equiv$  matrix-vector product complexity to get an  $\epsilon$ approximation.  Construction $\equiv$  complexity of the construction of order $m$ representation. Storage $\equiv$ cost of storage of a given representation. Adaptivity $\equiv$  ability to automatically adapt to different input operators.}\label{table:table}
\end{table}
\end{center}

\subsection{Extensions to higher dimensions}\label{subsec:extensionhigherdim}

Most of the results provided in this paper are based on standard approximation results in 1D, such as lemmas \ref{lem:approx_rate_Fourier}, \ref{lem:approx_rate_Wavelets} and \ref{lem:splineapprox}. 
All these lemma can be extended to higher dimension and we refer the interested reader to \cite{devore1993constructive,mallat1999wavelet,devore1998nonlinear,pinkus2012n} for more details. 

We now assume that $\Omega=[0,1]^d$ and that the diameter of the impulse responses is bounded by $\kappa\in [0,1]$. 
Using the mentioned results, it is straightforward to show that the approximation rate of all methods now becomes 
\begin{equation}
\|H-H_m\|_{HS} = O(\kappa^{d/2} m^{-s/d}). 
\end{equation}

The space $\Omega$ can be discretized  on a finite dimensional space of size $n^d$. 
Similarly, all complexity results given in table \ref{table:table} are still valid by replacing $n$ by $n^d$, $\epsilon^{-1/s}$ by $\epsilon^{-d/s}$ and $\kappa$ by $\kappa^d$.

\subsection{Extensions to least regular spaces}

Until now, we assumed that the TVIR $T$ belongs to Hilbert spaces (see e.g. assumption \ref{assumption1}). 
This assumption was deliberately chosen easy to clarify the presentation. 
The results can most likely be extended to much more general spaces using nonlinear approximation theory results \cite{devore1998nonlinear}.

For instance, assume that $T\in BV(\Omega\times \Omega)$, the space of functions with bounded variations. 
Then, it is well known (see e.g. \cite{cohen2003harmonic}) that $T$ can be expressed compactly on a Hilbert basis of tensor-product wavelets. 
Therefore, the convolution-product expansion \ref{eq:convolutionproduct} could be used by using the trick proposed in \ref{eq:trickMeyer}.

Similarly, most of the kernels found in partial differential equations (e.g. Calder\`on-Zygmund operators) are singular at the origin. 
Once again, it is well known \cite{meyer1995wavelets} that wavelets are able to capture the singularities and the proposed methods can most most likely be applied to this setting too.

A precise setting useful for applications requires more work and we leave this issue open for future work.

\subsection{Controls in other norms}

In all the paper we only controlled the Hilbert-Schmidt norm $\|\cdot\|_{HS}$. 
This choice simplifies the analysis and also allows getting bounds for the spectral norm 
\begin{equation}
 \|H\|_{2\to 2} = \sup_{\|u\|_{L^2(\Omega)}\leq 1} \|Hu\|_{L^2(\Omega)},
\end{equation}
since 
$\|H\|_{2\to 2} \leq \|H\|_{HS}$. In applications, it often make sense to consider other operator norms defined by
\begin{equation}
 \|H\|_{X\to Y} = \sup_{\|u\|_{X}\leq 1} \|Hu\|_{Y},
\end{equation}
where $\|\cdot \|_{X}$ and $\|\cdot \|_{Y}$ are norms characterizing some function spaces. We showed in \cite{escande2015sparse} that this idea could highly improve practical approximation results. 

Unfortunately, it is not clear yet how to extend the extend the proposed results and algorithms to such a setting and we also leave this question open for the future.

\section{Conclusion}\label{sec:conclusion}

In this paper, we analyzed the approximation rates and numerical complexity of convolution-product expansions.
This approach was shown to be efficient whenever the time or space varying impulse response of the operator is well approximated by a low rank tensor.
We showed that this situation occurs under mild regularity assumptions, making the approach relevant for a large class of applications. 
We also proposed a few original implementations of this methods based on orthogonal wavelet decompositions and analyzed their respective advantages precisely.
Finally, we suggested a few ideas to further improve the practical efficiency of the method.

\bibliographystyle{alpha}
\bibliography{biblio}

\end{document}